\documentclass[11pt,english]{article}
 \usepackage[english]{babel}
\usepackage{cite} 
 \usepackage[dvips]{graphicx}

\usepackage[usenames, dvipsnames]{xcolor}
\usepackage[utf8x]{inputenc}
\usepackage[T1]{fontenc}
\usepackage{tikz}
\usepackage{pgfplots}

\usepackage{amsfonts,amsmath,amsxtra,amsthm,amssymb,latexsym}
\newtheorem{theorem}{Theorem}[section]
 \newtheorem{corollary}[theorem]{Corollary}
 \newtheorem{lemma}[theorem]{Lemma}
 \newtheorem{proposition}[theorem]{Proposition}
 \theoremstyle{definition}
 \newtheorem{definition}[theorem]{Definition}
 \theoremstyle{remark}
 \newtheorem{remark}[theorem]{Remark}
 
 \numberwithin{equation}{section}

\newcommand{\R}{\mathbb R}
\begin{author}
{Jaime Angulo Pava and Andrés Gerardo Pérez Yépez  $^1$}
\end{author}

\begin{title}
{Existence and  orbital stability of standing-wave solutions of the NLS-log equation on a tadpole graph}
\end{title}
\date{}

\begin{document}
\maketitle

\centerline{$^1$ Department of Mathematics,
IME-USP}
 \centerline{Rua do Mat\~ao 1010, Cidade Universit\'aria, CEP 05508-090,
 S\~ao Paulo, SP, Brazil.}
 \centerline{\it andresgerardo329@gmail.com}
 \centerline{\it angulo@ime.usp.br}

\begin{abstract}
This work aims to study some dynamical aspects of the nonlinear logarithmic Schr\"odinger equation (NLS-log) on a tadpole graph, namely, a graph consisting of a circle with a half-line attached at a single vertex. By considering Neumann-Kirchhoff boundary conditions at the junction we show the existence and the orbital stability of standing wave solutions with a profile determined by a positive single-lobe state. Via a splitting- eigenvalue method, we identify the Morse index and the nullity index of a specific linearized operator around a  positive single-lobe state. To our knowledge, the results contained in this paper are the first to study the (NLS-log) on tadpole graphs. In particular, our approach has the prospect of being extended to study stability properties of other bound states for the (NLS-log) on a tadpole graph or other non-compact metric graph such as a looping-edge graphs.
\end{abstract}

\textbf{Key  words}. Nonlinear Schr\"odinger model, quantum graphs, standing wave solutions, stability, extension theory of symmetric operators, Sturm Comparison Theorem.


\section{Introduction}

The following Schr\"odinger model with a logarithmic non-linearity (NLS-log)
 \begin{equation}\label{NLSL}
i\partial_t u+\Delta u + uLog |u|^2=0,
\end{equation}
 where $u=u(x, t):\, \mathbb{R}^N\times\mathbb{R}\rightarrow \mathbb{C}$, $N\geq 1$, was introduced in 1976 by Bialynicki-Birula and Mycielski   \cite{BM} who proposed a model of nonlinear wave mechanics to obtain a nonlinear equation which helped to quantify departures from the strictly linear regime, preserving in any number of dimensions some fundamental aspects of quantum mechanics, such as separability and additivity of total energy of noninteracting subsystems. The NLS model in \eqref{NLSL} equation admits applications to dissipative systems \cite{HR}, quantum mechanics, quantum optics \cite{BSS}, nuclear physics \cite{H}, transport and diffusion phenomena (for example, magma transport) \cite{DFGL},
open quantum systems, effective quantum gravity, theory of superfluidity, and Bose-Einstein condensation (see \cite{H, Zlo10}  and the references therein).

The analysis of nonlinear evolution PDEs models on metric graphs has potential applicability in the analysis of physical models for modeling particle and wave dynamics in branched structures and networks. Since branched structures and networks appear in different areas of contemporary physics with many applications in electronics, biology, material science, and nanotechnology, the development of effective modeling tools is important for the many practical problems arising in these areas (see \cite{BeKu} and references therein). Nevertheless, real systems can exhibit strong inhomogeneities due to different nonlinear coefficients in different regions of the spatial domain or to a specific geometry of the spatial domain. Thus, we will chose a ``simple'' metric graph-tool such as the tadpole to discover several characteristics of the NLS-log model.

The  NLS-log ($N=1$) on metric graphs has been studied by several authors in the recent years  (see  \cite{AAr,  Ang2, AngGol17b, Ar0, Ar1} and reference therein). Two basic metric graphs $\Gamma_i$, $i=0,1$,  were examined. For $\Gamma_0=(-\infty, 0)\cup (0, +\infty)$ with boundary $\delta$-or  $\delta'$-interactions at the vertex $\nu=0$, the existence and orbital  (in)stability of standing wave solutions with a  Gausson profile were established. A similar study was also conducted for $\Gamma_1=\bigcup_{j=1}^N (0, +\infty)$, known as a star metric graph with the common vertex $\nu=0$. 

In this work, we study  issues related to  the existence and orbital stability of standing wave solutions of the  NLS-log ($N=1$)  on a tadpole graph, specifically the  vectorial model
\begin{equation}\label{nlslog}
   i\partial_t U + \Delta  U +  U \text{Log}| U|^2 =0.   
 \end{equation}
defined on a graph comprising a ring with one half-line attached at one vertex point (see Figure 1 below).
\begin{figure}[h]
 	\centering
\includegraphics[angle=0,scale=0.4]{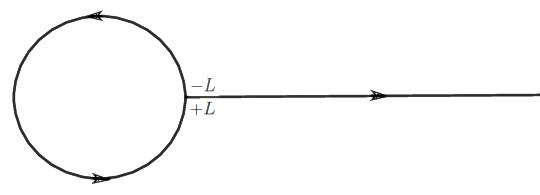} 
 	\caption{Tadpole graph}
\end{figure}


Thus, if in the tadpole graph, the ring is identified by the interval $[-L, L]$ and the semi-infinite line with $[L,+\infty)$, we obtain a metric graph $\mathcal{G}$ with a structure represented by the set $\mathbb E=\{e_0,e_1\}$ where $e_0 = [-L, L]$ and $e_1=[L,+\infty)$, which are the edges of $\mathcal{G}$ and they are connected at the unique vertex $v=L$. $\mathcal{G}$  is also called a lasso graph (see \cite{Exner} and references therein). In this form, we identify any function $U$ on $\mathcal{G}$ (the wave functions) with a collection $U=(u_e)_{e \in  E}$ of functions $u_e$ defined on the edge $e$ of $\mathcal{G}$. In the case of the NLS-log in (\ref{nlslog}), we have $U(x_e,t)=(u_e(x_e,t))_{e \in \mathbb E}$ and the nonlinearity $U \text{Log}|U|^2$, acting componentwise, i.e., for instance $(U \text{Log}|U|^2)_{e} = u_e \text{Log}|u_e|^2$. The action of the Laplacian operator $\Delta$ on the tadpole $\mathcal{G}$ is given by 
\begin{equation} \label{Lap}
    -\Delta : (u_e)_{e \in E} \rightarrow (-u''_e)_{e \in \mathbb E}. 
\end{equation}

Several domains make the Laplacian operator self-adjoint on a tadpole graph (see  Berkolaiko\&Kuchment \cite{BeKu}, Exner {\it et al.} \cite{Exner, ExJ} and Angulo\&Mun\text{$\bar{o}$}z \cite{AM} ). Here, we will consider a domain of general interest in physical applications. In fact, if we denote a wave function $U$ on the tadpole graph $\mathcal{G}$ as $U=(\Phi, \Psi)$, with $\Phi : [-L,L] \rightarrow \mathbb C$ and $\Psi : [L,+\infty ) \rightarrow \mathbb C$, we define  the following domains  for $-\Delta$: 
\begin{equation}\label{bcond}
    D_Z = \{U \in H^2(\mathcal{G}): \Phi(L)=\Phi(-L)=\Psi(L), \text{ and, } \Phi'(L)-\Phi'(-L)= \Psi'(L+) + Z\Psi(L) \},
\end{equation}
with $Z \in \mathbb R$  and  for any $n \ge 0$, $n \in \mathbb N$, 
\begin{equation*}
    H^n(\mathcal{G}) = H^n(-L,L) \oplus H^n(L,+\infty).
\end{equation*}
The boundary conditions in (\ref{bcond}) are called of $\delta$-interaction type if $Z \ne 0$, and of flux-balanced or Neumann-Kirchhoff condition if $Z=0$ (with always continuity at the vertex). It is not difficult to see that $(-\Delta, D_Z)_{Z \in \mathbb R}$ represents a one-parameter family of self-adjoint operators on the tadpole graph $\mathcal{G}$.  We note that is possible  determine other domains where the Laplacian is self-adjoint on a tadpole (see Berkolaiko\&Kuchment\cite{BeKu}). In particular, using  the approach in Angulo\&Mun\text{$\bar{o}$}z \cite{AM}, we can to obtain domains that can be characterized by the following family of 6-parameters of boundary conditions 
\begin{equation}\label{6bc}
     \begin{split}
         &\Phi(-L)=\Phi(L),\;\; (1-m_3)\Phi(L)=m_4\Psi(L)+m_5\Psi'(L),\\
         &\Phi'(L)-2\Phi'(-L)=m_1 \Phi(-L) + A_1\Psi(L)+ A_2\Psi'(L)\ \mbox{and }\\
         &\Phi'(L)-A_3\Psi'(L)=m_6\Phi(-L)+ m_7\Psi(L),
           \end{split}
 \end{equation}
 $A_1=m_4m_6-m_3m_7$, $A_2=m_5m_6-\frac{2m_3+m_7m_5m_3}{m_4}$ and  $A_3=\frac{2+m_7m_5}{m_4}$, $m_1, m_3, m_4, m_5, m_6, m_7\in \mathbb R$ and $m_4\neq 0$. Note that for $m_3=m_5=0$ and $m_4=1$, we get the so-called $\delta$-interaction conditions in \eqref{bcond}. Moreover, for $m_1=m_6=m_7=0$, $m_3=1$, $m_4=2$, and $m_5$ arbitrary, we get the following $\delta'$-interaction type condition on a tadpole graph 
 \begin{equation}\label{deltaprime}
     \begin{split}
         &\Phi'(-L)=\Phi'(L)=\Psi'(L),\;\; \Phi(L)=\Phi(-L)\;\; \mbox{and}\\
          &\Psi(L)=-\frac{m_5}{2}\Psi'(L).
                     \end{split}
 \end{equation}

Now, a problem of general interest is the interaction between standing waves in spatially confined systems and those in large or unbounded reservoirs (so we can say that a tadpole is a configuration that fulfills these characteristics). Our main interest here will be to study some dynamics aspects of (\ref{nlslog}) such as the existence and orbital stability of standing wave solutions given by the profiles $U(x,t)= e^{ict} \Theta(x)$, with $c \in \mathbb R$, $\Theta=(\phi, \psi) \in D_Z$, $Z=0$, and satisfying the NLS-log vectorial equation 

\begin{equation}\label{nlslogv}
    -\Delta \Theta +c\Theta - \Theta \text{Log}(|\Theta|^2)=0. 
\end{equation}
More explicitly, for  $\phi$ and $\psi$ real-valued we obtain the following system, one on the ring and the other one on the half-line respectively, 

\begin{equation}\label{sistema}
 \left\{ \begin{array}{ll}
  -\phi''(x)+ c \phi(x)-\text{Log}(\phi^2(x))\phi(x)=0, \;\;\;\;x\in (-L,L),\\
   -\psi''(x)+ c \psi(x)-\text{Log}(\psi^2(x))\psi(x)=0, \;\;\;\;x\in (L,+\infty),\\
\phi( L)=\phi(-L) =\psi(L),\\
\phi'( L)-\phi'(-L)=\psi'(L+).
  \end{array} \right.
 \end{equation}

The critical challenge in solving  (\ref{sistema}) lies in the component $\phi$ on $[-L, L]$ (we note that explicit profiles for $\phi$ are not known). The component of $\psi$ is given by a translated Gausson-profile (see \cite{Caz83, AC}) of the form  

\begin{equation}\label{perfil}
    \psi_c(x)= e^{\frac{c+1}{2}} e^{\frac{-(x-L+a)^2}{2}}, \ \ \  a \ne 0, \ \ \ c \in \mathbb R, \ \ \ x \geq L.
\end{equation}

Among all profiles for (\ref{sistema}) (see, for instance, Figures 2 and 3 below for the case of $\phi$- profiles), we focus on \textit{positive single-lobe states}. More precisely, we define (see Figure 2).
 
\begin{definition}
    The  standing wave profile $\Theta = (\phi, \psi) \in D_0$ is said to be a positive single-lobe state for (\ref{sistema}) if each component is positive on evry edge of $\mathcal{G}$, the maximum of $\Theta$ is achieved at a single internal point symmetrically located on $[-L, L]$, and $\phi$ is monotonically decreasing on $[0, L]$. Moreover, $\psi$ is strictly decreasing on $[L,+\infty)$. 
\end{definition}

\begin{figure}[h]
 	\centering
\includegraphics[angle=0,scale=0.6]{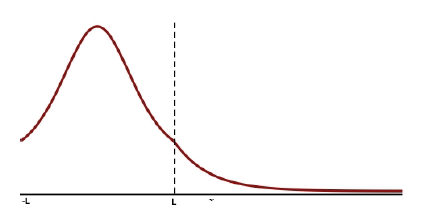} 
 	\caption{A positive single-lobe state profile for the NLS-log model on $\mathcal{G}$}
\end{figure}

The study of   ground or bound states on general metric graphs  for the NLS model with a polynomial nonlinearity
\begin{equation}\label{nlsp}
     i\partial_t U + \Delta  U + | U|^{2p}  U =0, \quad p>0
 \end{equation}
has been investigated in  \cite{ACFN, AC1, AC2, AST, ASTthr, ASTcri, CaFi, NPS, KMP, KNP, KP, NP, symme}. To our knowledge, the results contained in this paper are the first in studying the NLS-log on tadpole graphs.

Our focus in this work is to study the existence and the orbital stability for the NLS-log model of positive single-lobe states in the case $Z=0$ (for the case $Z\neq 0$, we refer  the reader to section 7 below). For the existence, we use  tools from dynamical systems theory for orbits on the plane, based on  the period function introduced in \cite{KMP, KP}. To our knowledge, this approach has not been applied in the literature to logarithmic nonlinearities,  and a non-trivial  analysis will be required. For the stability, we follow the abstract stability framework by Grillakis\&Shatah\&Strauss \cite{GrilSha87}. For clarity, we outline the main steps of this framework for standing wave solutions for NLS-log models on a tadpole graph (see Theorem \ref{main} in the Appendix below). Subsequently, we will present our main results. 

To begin,  we note that the basic symmetry associated to the NLS-log model \eqref{nlslog} on tadpole graphs is the phase invariance:  if $ U$ is a solution of \eqref{nlslog}, then $ e^{i\theta}U$ is also a solution for any $\theta\in [0,2\pi)$. Thus, we  define orbital stability for \eqref{nlslog}  as follows (see \cite{GrilSha87}).

\begin{definition}\label{dsta}
The standing wave $U(x,t) = e^{ic
t}(\mathbf{\phi}(x),  \mathbf{\psi}(x))$ is said to
be \textit{orbitally stable} in a Banach space $X$ if for any
$\varepsilon > 0$
there exists $\eta > 0$ with the following property: if
$U_0 \in X$
satisfies $||U_0-(\Phi, \Psi)||_{X} <\eta$,
then the solution $U(t)$ of (\ref{nlslog}) with $U(0) = U_0$
exists for any
$t\in\mathbb{R}$ and
\[\sup\limits_{t\in \mathbb{R}
}\inf\limits_{\theta\in [0, 2\pi)}||U(t)-
e^{i\theta}(\mathbf{\phi}, \mathbf{\psi})||_{X} < \varepsilon.\]
Otherwise, the standing wave $U(x,t) = e^{ic t}(\mathbf{\phi}(x),  \mathbf{\psi}(x))$ is
said to be \textit{orbitally unstable} in $X$.
\end{definition}

The space $X$ in Definition \ref{dsta} for the model \eqref{nlslog} will depend on the domain of the action of $-\Delta$, namely, $D_{0}$, and a specific weighted space. Indeed, we will consider  the following spaces 
 \begin{equation}\label{E}
\mathcal E(\mathcal G) =\left\{(f,  g)\in H^ 1(\mathcal G):  f(-L)=f(L)=g(L)\right\} \quad\text{``continuous energy-space''},
\end{equation}
and the Banach spaces $W(\mathcal{G})$ and $\widetilde{W}$   defined by 
\begin{equation}\label{WW}
\begin{aligned}
    & W(\mathcal{G})=\{(f,g) \in \mathcal{E}(\mathcal{G}) :  |g|^2\text{Log}|g|^2 \in L^1(L,+\infty)\}\\
   & \widetilde{W}= \{(f,g) \in \mathcal{E}(\mathcal{G}): xg \in L^2(L,+\infty)\}.
\end{aligned}
\end{equation}
We note that $\widetilde{W} \subset W(\mathcal{G})$ (see Lemma \ref{W} below). Due to our local stability analysis  (not variational type, see section 7 for discussion on this approach), we will consider $X=\widetilde{W} $ in Definition \ref{dsta}.

Next, we consider the following two  functionals associated with \eqref{nlslog} 
 \begin{equation}\label{Ener}
E( U)=\|\nabla U\|^2_{L^2(\mathcal G)}- \int_{-L}^L |f_1|^2 \text{Log}(|f_1|^2)dx - \int_{L}^{+\infty} |f_2|^2\text{Log}(|f_2|^2)dx,\qquad (\text{energy})
\end{equation}
 and 
 \begin{equation}\label{mass}
Q( U)=\|  U\|^2_{L^2(\mathcal G)}, \qquad\qquad\qquad\qquad\qquad\qquad\qquad\;\;\; \quad\qquad\quad\;\;\;(\text{mass})
\end{equation}
where $ U=(f_1, f_2)$. These functionals satisfy $E, Q\in C^1(W(\mathcal{G}): \mathbb R)$ (see \cite{Caz83} or Proposition 6.3 in \cite{AAr}) and, at least formally, $E$ is conserved by the flow of \eqref{nlslog}. The use of the space $W(\mathcal{G})$ is because $E$ fails to be continuously differentiable on $\mathcal{E}(\mathcal{G})$ (a proof of this can be based on the ideas in \cite{Caz83}). Now, as our stability theory is based on the framework of Grillakis {\it et al.} \cite{GrilSha87}, $E$ needs to be twice continuously differentiable at the profile $\Theta\in D_0$. To satisfy this condition we introduced the space $\widetilde{W}$. Moreover, this space naturally appears in the following study of the linearization of the action functional around $\Theta$. We note that $E, Q\in C^1(\widetilde{W})$.

Now, for a fixed $c \in \mathbb R$,  let $U_c(x,t) = e^{ic t}(\mathbf{\phi}_c(x), \mathbf{\psi}_c(x))$ be a standing wave solution for  \eqref{nlslog} with $(\phi_c, \psi_c)\in D_{0}$ being a  positive single-lobe state. Then, for the action functional 
\begin{equation}\label{S1}
\mathbf S(U)=E(U)-(c+1) Q(U),\quad\quad U\in \widetilde{W},
\end{equation}
we have  $\mathbf S'(\phi_c, \psi_c)=\mathbf 0$.  Next, for $U=U_1+iU_2$ and $W=W_1+iW_2$, where the functions $U_j$, $W_j$, $j=1,2$, are real. The second variation of $\mathbf S$ in $(\phi_c, \psi_c)$ is  
\begin{equation}\label{S2}
\mathbf S''(\phi_c, \psi_c) (U, W)=\langle \mathcal L_1 U_1, W_1\rangle+\langle\mathcal L_2 U_2, W_2\rangle,
\end{equation}
where the two $2\times 2$-diagonal operators $\mathcal L_{1}$ and $\mathcal L_{2}$  are given for 
\begin{equation}\label{L+}
 \begin{array}{ll}
 &\mathcal{L}_1= diag\left( - \partial_x^2 + (c-2)- \text{Log}|\phi_c|^2, - \partial_x^2 + (c-2)- \text{Log}|\psi_c|^2 \right) \\
 \\
        &\mathcal{L}_2= diag\left( - \partial_x^2 + c- \text{Log}|\phi_c|^2,  - \partial_x^2 + c- \text{Log}|\psi_c|^2 \right)
\end{array} 
 \end{equation}
These  operators are self-adjoint with domain (see Theorem \ref{Oper} below)
$$
\mathcal D:= \{ (f,g) \in D_0: x^2g \in L^2([L,+\infty)) \}.
$$
Since $(\phi_c, \psi_c)\in \mathcal  D$ and satisfies system \eqref{sistema},  $\mathcal L_2(\phi_c, \psi_c)^t= \mathbf 0$,  so the kernel of $\mathcal L_2$ is non-trivial. Moreover, $\langle \mathcal{L}_1(\phi_c, \psi_c)^t, (\phi_c, \psi_c)^t \rangle <0$ implies that the Morse index of $\mathcal{L}_1$, $n(\mathcal{L}_1)$, satisfies $n(\mathcal{L}_1)\geqq 1$. Next, from  \cite{GrilSha87} we know that the Morse index and the nullity index of the operators $\mathcal L_1$ and $\mathcal L_2$ are a fundamental step in deciding about the orbital stability of standing wave solutions.  For the case of the profile $(\mathbf{\phi}_c, \mathbf{\psi}_c)$ being a positive single-lobe state, our  main results are the following:

 \begin{theorem}\label{FrobeN}  Consider the self-adjoint operator $(\mathcal L_1, \mathcal D)$ in \eqref{L+} determined by the positive single-lobe state $(\mathbf{\phi}_c, \mathbf{\psi}_c)$.  Then,
  \begin{enumerate}
  \item[1)] Perron-Frobenius property: let $\beta_0<0$ be the smallest eigenvalue of $\mathcal L_1$ with associated eigenfunction $(f_{\beta_0}, g_{\beta_0})$. Then, $f_{\beta_0}$ is positive and even on $[-L,L]$, and $g_{\beta_0}(x)>0$ with $x\in[L, +\infty)$,
  \item[2)]  $\beta_0$ is simple,
 \item[3)] the Morse index of $\mathcal L_1$ is one,
 \item[4)] The kernel of $\mathcal L_{1}$ on $\mathcal D$ is trivial. 
 \end{enumerate}
\end{theorem}

\begin{theorem}\label{L2} Consider the self-adjoint operator $(\mathcal L_2, \mathcal D)$ in \eqref{L+} determined by the positive single-lobe state $(\mathbf{\phi}_c, \mathbf{\psi}_c)$.  Then,
  \begin{enumerate}
  \item[1)] the kernel of $\mathcal L_2$, $ker(\mathcal L_2)$, satisfies $ker(\mathcal L_2)=span\{(\mathbf{\phi}_c, \mathbf{\psi}_c)\}$.
  \item[2)]  $\mathcal L_2$ is a non-negative operator, $\mathcal L_2\geqq 0$.
  \end{enumerate}
\end{theorem}

The proof of Theorem \ref{FrobeN} will be based on a  {\it{splitting eigenvalue method}} applied to  $\mathcal L_1:=\text{diag}(\mathcal L_1^0, \mathcal L_1^1)$ in \eqref{L+} on a tadpole graph (see Lemma \ref{split}). More precisely, we reduce the eigenvalue problem associated with $\mathcal L_1$ with domain $\mathcal D$   to two classes of eigenvalue problems, one with periodic boundary conditions on $[-L, L]$ for $\mathcal L_1^0$ and the other one with Neumann boundary conditions on  $=[L, +\infty)$ for $\mathcal L_1^1$. Thus,  by using tools of the extension theory of Krein-von Neumann for symmetric operators, the theory of real coupled self-adjoint boundary conditions on $[-L, L]$ and the Sturm Comparison Theorem, will lead to our results.

Our orbital stability result  is the following:
\begin{theorem}\label{stability}
    Consider $Z=0$ in (\ref{sistema}). Then, there exists a $C^1$-mapping $c \in \mathbb R \rightarrow \Theta_c= (\phi_c, \psi_c)$ of positive single-lobe states on $\mathcal{G}$. Moreover, for $c \in \mathbb R$,  the orbit 
$$
\{e^{i\theta} \Theta_c : \theta \in [0,2\pi) \}
$$
is stable.  
\end{theorem}

The proof of Theorems \ref{FrobeN}-\ref{L2} are given in Section 4. The  orbital stability statement follows from Theorems \ref{FrobeN}-\ref{L2} and  the abstract stability framework  by Grillakis\&Shatah\&Strauss in \cite{GrilSha87}.  For the convenience of the reader, we provide an adaptation of the abstract results in \cite{GrilSha87} to the case of tadpole graphs in Theorem \ref{main} (Appendix). We note that the main challenge in  applying  Theorem \ref{main} lies  in spectral analysis, as the Vakhito-Kolokolov condition $\frac{d}{dc} \|(\phi_c, \psi_c)\|^2>0$ is essentially trivial for the NLS-log model. Definition 1.2 requires {\it a priori} information about the local/global well-posedness of the Cauchy problem for \eqref{nlslog}, which is established in Section 2 for   the space $\widetilde{W}$ in \eqref{espaceW}.   

The  existence of a  $C^1$-mapping  for positive single-lobe states in Theorem \ref{stability} relies on  dynamical systems theory for planar orbits  via the period function for second-order differential equations (see  \cite{AC1, AC2, KNP, KMP, KP, NP} and reference therein)

We would like to point out that our approach for studying positive single-lobe states for the NLS-log on a tadpole has prospects of being used  to study other standing wave profiles, such as positive two-lobe states (see  Angulo \cite{AC1} and Figure 3) or the NLS-log on others metric graphs such as looping edge graphs, a graph consisting of a circle with several half-lines attached at a single vertex (see Figure 7 in Section 7 ). 
 \begin{figure}[h]
 	\centering
\includegraphics[angle=0,scale=0.4]{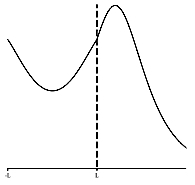}
\caption{A positive two-lobe state profile on a tadpole graph}
\end{figure}

The paper is organized as follows.  In Section 2,  we establish  local and global well-posedness results for the NLS-log model on a tadpole graph. In Section 3,  we show the linearization of the NLS-log around a positive single-lobe state and its relation to the self-adjoint operators $\mathcal L_1, \mathcal L_2$. In Section 4,  we show Theorems \ref{FrobeN}-\ref{L2}  via our splitting eigenvalue lemma (Lemma \ref{split}). In Section 5, we provide the proof of the existence of positive single-lobe states and their stability under the NLS-log flow. In the  Appendix, we briefly outline   tools from  Krein-von Neumann extension theory, a Perron-Frobenius property for $\delta$-interaction Schr\"odinger operators on the line, and the orbital stability criterion from Grillakis\&Shath\&Strauss  \cite{GrilSha87} adapted to our framework.

\vskip0.2in

\noindent \textbf{Notation.} Let $-\infty\leq a<b\leq +\infty$. We denote by $L^2(a,b)$  the  Hilbert space equipped with the inner product $(u,v)=\int\limits_a^b u(x)\overline{v(x)}dx$.
 By $H^n(\Omega)$  we denote the classical  Sobolev spaces on $\Omega\subset \mathbb R$ with the usual norm.   We denote by  $\mathcal{G}$ the tadpole graph parametrized by the set of edges
  $ \mathbb {E} =\{e_0, e_1\}$, where $e_0=[-L,L]$ and  $e_1=[L, +\infty)$, and attached to the common vertex $\nu=L$. On the graph $\mathcal{G}$  we define the  spaces 
  $$
 L^ p(\mathcal G)=L^ p(-L, L) \oplus L^p( L, +\infty), \;\; \,p>1,
 $$   
with the natural norms. Also, for $U= (u_1, g_1), V= (v_1, h_1)\in L^2(\mathcal G)$, the inner product on $L^2(\mathcal G)$ is defined by
$$
\langle U, V \rangle=\int_{-L}^L u_1(x)\overline{v_1(x)}dx +  \int_L^{+\infty} g_1(x)\overline{h_1(x)}dx. 
$$
 Let $A$ be a closed densely defined symmetric operator in the Hilbert space $H$. The domain of $A$ is denoted by $D(A)$. The deficiency indices of $A$ are denoted by  $n_\pm(A):=\dim ker(A^*\mp iI)$, with $A^*$ denoting the adjoint operator of $A$.  The Morse index of $A$, denoted by  $n(A)$, is the number of negative eigenvalues counting multiplicities.


\section{Global well-posedness in  $\widetilde{W}$}

In this section, we show that the Cauchy problem associated with the  NLS-log model on a tadpole graph is globally well-posed in the space $\widetilde{W}$. From Definition \eqref{dsta} this information is crucial for the stability theory. We start with the following technical result.

\begin{lemma}\label{W}
Let $W(\mathcal{G})$ and $\widetilde{W}$ be the Banach spaces defined by 

\begin{equation}\label{espaceW}
\begin{aligned}
    & W(\mathcal{G})=\{(f,g) \in \mathcal{E}(\mathcal{G}) : |g|^2\text{Log} |g|^2 \in L^1(L,+\infty)\}, \\
   & \widetilde{W}= \{(f,g) \in \mathcal{E}(\mathcal{G}): xg \in L^2(L,+\infty)\}.
\end{aligned}
\end{equation}
Then $\widetilde{W} \subset W(\mathcal{G})$.
    
\end{lemma}
\begin{proof}
\begin{itemize}
\item[1)] Let's start by proving that $\widetilde{W} \subset L^1(\mathcal{G})$. 
For $(\varphi,\xi) \in \widetilde{W}$ and the Cauchy-Schwarz inequality we have 

\begin{equation*}
\begin{aligned}
    \int_{-L}^L |\varphi|dx + \int_L^{+\infty} |\xi| dx &=  \int_{-L}^L |\varphi|dx + \int_L^{+\infty} x|\xi|\frac{1}{x} dx \\
    &\le 2L\sup_{[-L,L]} |\varphi| +  \left( \int_L^{+\infty} x^2|\xi|^2 dx \right)^\frac{1}{2}  \left( \int_L^{+\infty} \frac{1}{x^2} dx \right)^\frac{1}{2} <\infty. 
\end{aligned}
\end{equation*}
\item[2)] Let again $(\varphi,\xi) \in \widetilde{W}$, then 
\begin{equation}\label{radio1}
    \int_L^{+\infty} |\xi|^2 | \text{Log} |\xi|| dx =  \int_{\{x \in [L,+\infty): |\xi(x)|<1\}} |\xi|^2| \text{Log}|\xi|| dx +  \int_{\{x \in [L,+\infty): |\xi(x)|\ge 1\}} |\xi|^2| \text{Log}|\xi|| dx.
\end{equation}

Note also that 
\begin{equation}\label{plog}
    | \text{Log}|\xi (x)||<\frac{1}{|\xi(x)|} \ \ \text{for} \ \ |\xi(x)|<1, \ \ \text{and} \ \ | \text{Log}|\xi(x)||< |\xi(x)|  \ \ \text{for} \ \ |\xi(x)|\ge 1.
\end{equation}

Since $\xi \in H^1([L,+\infty))$, there exists $k>L$ such that $|\xi|<1$ for $[L,+\infty)\setminus [L,k]$. Thus, from (\ref{radio1}), (\ref{plog}), and the inclusion $\widetilde{W} \subset L^1(\mathcal{G})$ we get 

\begin{equation*}
    \begin{aligned}
        \int_L^{+\infty} |\xi|^2|Log|\xi|| dx &\le  \int_{\{x \in (k,+\infty): |\xi|<1\}} |\xi| dx + \int_{\{x \in [L,k]: |\xi|< 1\}} |\xi| dx +\int_{\{x \in [L,k]: |\xi|\ge 1\}} |\xi|^3 dx  \\
        &\le \int_{\{x \in (k,+\infty): |\xi|<1\}} |\xi| dx +(k-L)+(k-L)\sup_{[L,k]} |\xi|^3 < \infty.
    \end{aligned}
\end{equation*}
\end{itemize}
The assertion is proved. 
\end{proof}

\begin{theorem}\label{global}
For any $U_0\in \widetilde{W}$, there is a unique solution $U \in C(\mathbb R,\widetilde{W})$ of (\ref{nlslog}) such that $U(0)=U_0$.  Furthermore, the conservation of energy and mass hold, {\it i.e}, for any $t\in \mathbb R$, we have $E(U(t))=E(U_0)$ and $Q(U(t))=Q(U_0)$.
\end{theorem}

\begin{proof}
The idea of the proof is to use the strategy proposed in \cite{Caz80}-section 2 (see also \cite{Caz03, Caz98})  and adapted to the NLS-log model on the tadpole graph. As the analysis follows the Cazenave's framework which involves multiples steps, we  highlight the key modifications for our case. In this way, we introduce the ``reduced'' Cauchy problem
\begin{equation}\label{reduced}
 \left\{ \begin{array}{ll}
    i\partial_t U_n + \Delta  U_n+  U_n  F_n (| U_n|^2)=  0,\\
   U_n(0)=  U_0\in \widetilde{W},
  \end{array} \right.
 \end{equation}
where for $ U_n=(u_n,v_n)$, $ U_n  F_n (| U_n|^2)=(u_n f_n(|u_n|^2), v_n f_n(|v_n|^2))$,  $f_n(s)=\inf\{n, \text{sup}\{-n, f(s)\}\}$, and $f(s)=\text{Log}(s)$, $s>0$. Initially, we show that the problem in \eqref{reduced} has a unique global solution for any $ U_0\in \mathcal E(\mathcal G)$ and fixed $n$. To use  Theorem 3.3.1 in \cite{Caz03}, define the mapping  $g_n: L^2(\mathcal G)\to L^2(\mathcal G)$, $g_n(U)=U  F_n (| U|^2)$, which is  Lipschitz continuous on bounded sets of $L^2(\mathcal G)$ (since  each $f_n$ is Lipschitz  on $\mathbb R^+$). For $p_n(s)=\int_0^s f_n(s)ds$, define $P_n(w)=p_n(|w|^2)$, yielding  $P_n'(w)=|w| f_n(|w|^2)$  formally. For $G_n(U)= (P_n(u), P_n(v))$ with $U=(u,v)\in \mathcal E(\mathcal G)$,  we get $G_n'(U)=g_n(U)$ (where the derivative of $G_n$ is computed component-wise). Since $ \Delta $ with domain $D_0$ is  self-adjoint and  non-positive operator in $L^2(\mathcal G)$ (see Remark \ref{spectrum} below),   Theorem 3.3.1 implies  that for any $ U_0\in \mathcal E(\mathcal G)$ there exists a unique global solution $U_n $ of \eqref{reduced} such that $U_n\in C(\mathbb R,  \mathcal E(\mathcal G))\cap C^ 1(\mathbb R,  \mathcal E(\mathcal G)' ) $. Furthermore, for \eqref{reduced}, the conservation of  energy and charge hold, {\it i.e.}, for all $t\in [-T, T]$
\begin{equation}\label{Ener_n}
\begin{array}{ll}
&Q(U_n(t))= \|  U_n(t)\|^2_{L^2(\mathcal G)}= \|  U_0\|^2_{L^2(\mathcal G)},\quad  E_n(U_n(t))= E_n(U_0),\\
&E_n( U)=\|\nabla U\|^2_{L^2(\mathcal G)}- \int_{-L}^L P_n(|u|^2)dx - \int_{L}^{+\infty} P_n(|v|^2) dx,\qquad U=(u,v),
\end{array} 
 \end{equation}
 and so we also obtain that 
 \begin{equation}\label{uniforme}
  \|U_n\|_{L^\infty([-T, T],  \mathcal E(\mathcal G))}\leqq C,\quad\text{for all}\;\; n.
 \end{equation}
Note that the last statement is a consequence of $U_0\in \widetilde{W}\subset W(\mathcal G)$ by Lemma \ref{W} and from Lemmas 2.3.2, 2.3.3, and 2.3.4 in \cite{Caz80}. Now, for $\mathcal C_R=\{ (f, g)\in L^2 ((-L, L)\times (L, R)): f(-L)=f(L)=g(L)\}$ with $R>L$, we consider the Hilbert spaces $\mathcal E_R=\{ (p,q): (p,q)\in (H^1(-L, L)\times H^1(L, R))\cap \mathcal C_{R}\}$ and the dual space of $\mathcal E_R$, $\mathcal E_R'$. We note that as for any $a, b\in \mathbb R$, the embedding $H^1(a, b) \hookrightarrow C([a,b])$ is compact, we get initially that $\mathcal C_{R}$ is closed in $H^1(-L, L)\times H^1(L, R)$ and so  $\mathcal E_R$ is a Hilbert space. Further,  from the embedding relations  $\mathcal E_R \underset{compact} {\hookrightarrow} \mathcal C_R \hookrightarrow\mathcal E_R'$ follows by Aubin-Lions theorem (\cite{lions})  implies that
$$
\mathcal J=\{F: F\in L^2([-T, T], \mathcal E_R), F'=\frac{d}{dt}F \in L^2([-T, T], \mathcal E_R ' )\},
$$
is a Banach space compactly embedded in $L^2([-T, T], \mathcal C_R)$. By following a similar  analysis to Lemma 2.3.5 in \cite{Caz80} (see also  Lemma 9.3.6 in \cite{Caz03}), we get  that $U_{n_k}\to U$ in $L^\infty([0, T],  \mathcal E(\mathcal G))$ weak-$\star$. Therefore,  we get that $U$ is a solution of \eqref{nlslog} in the sense of distributions. Moreover, the conservation of charge $Q$ in \eqref{mass}  holds and so $U_{n_k}(t)\to U(t)$ strong  in  $ \mathcal C=\{ (f, g)\in L^2 ((-L, L)\times (L, +\infty)): f(-L)=f(L)=g(L)\}$. We have also that energy $E$ in \eqref{Ener}  is conserved via a standard monotonicity argument (see section 2.4 in \cite{Caz80}). Thus, the inclusion $U\in C(\mathbb R; \mathcal E(\mathcal (\mathcal G))$ follows from conservation laws. Lastly,  for $U_0=(g,h)$, the condition $xh\in L^2(L,+\infty)$ implies $xv\in L^2(L,+\infty)$ (for the solution $U=(u,v)$), repeating the argument of Lemma 7.6.2 in \cite{Caz98}. This finishes the proof.
\end{proof}

\begin{remark}\label{spectrum}
By considering $(-\Delta_Z, D_Z)$, with $-\Delta_Z=-\Delta$ and $D_Z$ in \eqref{bcond}, it is possible to see the following (see Angulo\&Plaza \cite{AP}, Noja\&Pelinovsky \cite{NP}): for every $Z \in \R$, the essential spectrum of the self-adjoint operator $-\Delta_Z$ is purely absolutely continuous and $\sigma_{ess}(-\Delta_Z) = \sigma_{ac}(-\Delta_Z) = [0,+\infty)$. If $Z > 0$, then $-\Delta_Z$ has exactly one negative eigenvalue, {\it i.e.}, its point spectrum is $\sigma_{pt}(-\Delta_Z) = \{ - \varrho_{_{Z}} ^2 \}$, where $\varrho_{_{Z}} > 0$ is the only positive root of the transcendental equation
\[
\varrho \left( 2 \tanh \Big(\varrho L\Big) + 1 \right) - Z = 0,
\]
and with the eigenfunction
\[
\boldsymbol{\Phi}_Z = \begin{pmatrix}
\cosh (x\varrho_{_{Z}}),\;\;x\in [-L,L]  \\
\\ e^{-x\varrho_{_{Z}}},\;\;\;\;\;\;\;x\in [L,+\infty)
\end{pmatrix}.
\]
If $Z \leqq 0$, then $-\Delta_Z$ has no point spectrum, $\sigma_{pt}(-\Delta_Z) = \varnothing$. Therefore, for $Z=0$ we have that $-\Delta_Z$ is a non-negative operator.
\end{remark}

\section{Linearization of NLS-log equation}

For fixed $c \in \mathbb R$, let $U(x,t)=e^{ict}\Theta_c(x)$ be a standing wave solution for (\ref{nlslog}) with $\Theta_c(x)=(\phi_c(x),\psi_c(x)) \in D_0$ being a positive single-lobe state. We consider the action functional
$S_c=E+(c+1)Q $, thus $\Theta_c$ is a critical point of $S_c$. Next, we determine the linear operator associated with the second variation of $S_c$ calculated at $\Theta_c$, which is crucial for applying the approach in \cite{GrilSha87}. To express $S''_c(\Theta_c)$, it is convenient to split $u,v \in \widetilde{W}$ into real and imaginary parts: $u=u_1 + i u_2$, $v= v_1 + iv_2$. Then, we get $S''_c(\Theta_c)(u,v)$ can be formally rewritten as 
\begin{equation}\label{segd}
\begin{aligned}
    S''_c(\Theta_c)(u,v) =B_1(u_1,v_1) + B_2(u_2,v_2),
\end{aligned}
\end{equation}
where

\begin{equation}\label{fbi}
\begin{aligned}
    B_1((f,g),(h,q))=& \int_{-L}^L f' h' dx + \int_L^{+\infty} g' q' dx +  \int_{-L}^L f h[c-2-\text{Log} |\phi_c|^2 ] dx    \\ &+ \int_L^{+\infty} g q[(x-L+a)^2 -3] dx, \\
    B_2((f,g),(h,q))=& \int_{-L}^L f' h' dx + \int_L^{+\infty} g' q' dx +  \int_{-L}^L f h[c-\text{Log} |\psi_c|^2 ] dx    + \int_L^{+\infty} g q[(x -L+a)^2 -1] dx,
\end{aligned}
\end{equation}
and dom$(B_i)=\widetilde{W} \times \widetilde{W}$, $i \in \{1,2\}$. Note that the forms $B_i$, $i \in \{1,2\}$, are bilinear bounded from below and closed. Therefore, by the First Representation Theorem (see \cite{K}, Chapter VI, Section 2.1),  they define operators $\widetilde{\mathcal{L}}_1$ and $\widetilde{\mathcal{L}}_2 $ such that for $i \in \{1,2\}$

\begin{equation}\label{ope}
\begin{aligned}
    &\text{dom}(\widetilde{\mathcal{L}}_i)= \{ v \in \widetilde{W}: \exists w \in L^2(\mathcal{G}) \text{  s.t.  } \forall z \in \widetilde{W}, B_i(v,z)=\langle w,z\rangle\}\\
    &\widetilde{\mathcal{L}}_i v= w.
\end{aligned}
\end{equation}

\begin{theorem} \label{Oper}
    The operators $\widetilde{\mathcal{L}}_1$ and $\widetilde{\mathcal{L}}_2 $ determined in (\ref{ope}) are given by 
    \begin{equation*}
    \begin{aligned}
        &\widetilde{\mathcal{L}}_1= \text{diag}\left( - \partial_x^2 + (c-2)- \text{Log}|\phi_c|^2, -\partial_x^2 +(x-L+a)^2 -3 \right) \\
        &\widetilde{\mathcal{L}}_2= \text{diag}\left( - \partial_x^2 + c-\text{Log}|\phi_c|^2, -\partial_x^2 +(x-L+a)^2 -1 \right)
    \end{aligned}
    \end{equation*}
on the domain $\mathcal D= \{ (f,g) \in D_0: x^2g \in L^2(L,+\infty) \}$. Thus, $\widetilde{\mathcal{L}}_i=\mathcal{L}_i$ defined in \eqref{L+}.
\end{theorem}

\begin{proof}
Since the proof for $\widetilde{\mathcal{L}}_2$ is similar to the one for $\widetilde{\mathcal{L}}_1$, we deal with $\widetilde{\mathcal{L}}_1$. Let $B_1=B^0 + B^1$, where $B^0: \mathcal{E}(\mathcal{G}) \times \mathcal{E}(\mathcal{G}) \rightarrow \mathbb R $ and $B^1: \widetilde{W} \times \widetilde{W}  \rightarrow \mathbb R$ are defined by 

\begin{equation}
    B^0\left( (f,g), (h,q) \right)= \langle (f',g'),(h',q') \rangle, \ \ B^1\left((f,g), (h,q)\right)= \langle V_1(f,g),(h,q) \rangle,
\end{equation}
and $V_1(f,g)= ([c-2-\text{Log}|\phi_c(x)|^2]f, [(x-L+a)^2 - 3]g)$. We denote by $\mathcal{L}^0$ (resp. $\mathcal{L}^1$) the self-adjoint operator on $L^2(\mathcal{G})$ associated (by the First Representation Theorem) with $B^0$ (resp. $B^1$). Thus, 

\begin{equation*}
\begin{aligned}
    &\text{dom}(\mathcal{L}^0)= \{ v \in \mathcal{E}(\mathcal{G}): \exists w \in L^2(\mathcal{G}) \text{  s.t.  } \forall z \in \mathcal{E}(\mathcal{G}), B^0(v,z)=\langle w,z\rangle\}\\
    &\mathcal{L}^0v= w.
\end{aligned}
\end{equation*}
We claim that $\mathcal{L}^0$ is the self-adjoint operator 
\begin{equation*}
   \mathcal{L}^0= -\Delta = -\frac{d^2}{dx^2}, \ \ \text{dom}(-\Delta)=D_0. 
\end{equation*}
Indeed, let $v=(v_1,v_2) \in D_0$ and $w= -v'' \in L^2(\mathcal{G})$. Then for every $z=(z_1,z_2) \in \mathcal{E}(\mathcal{G})$ and using integration by parts, we have

\begin{equation*}
\begin{aligned}
    B^0(v,z) &= (v',z')=\int_{-L}^L v_1' z_1' dx + \int_L^{+\infty} v_2' z_2' dx \\ 
    &= v_1'(L)z_1(L) - v_1'(-L)z_1(-L) - v_2'(L)z_2(L) - \int_{-L}^L v_1'' z_1' dx - \int_L^{+\infty} v_2'' z_2' dx \\ 
    &= (-v'', z)= (w,z).
\end{aligned}
\end{equation*}
Thus, $v \in$ dom$(\mathcal{L}^0)$ and $\mathcal{L}^0v= w=-v''=-\Delta v$. Hence, $-\Delta \subset \mathcal{L}^0$. Since $-\Delta $ is self-adjoint on $D_0$, $\mathcal{L}^0= -\Delta$. 

Again, by the First Representation Theorem, 
\begin{equation*}
\begin{aligned}
    &\text{dom}(\mathcal{L}^1)= \{ v \in \widetilde{W}: \exists w \in L^2(\mathcal{G}) \text{  s.t.  } \forall z \in \widetilde{W}, B^1(v,z)=\langle w,z\rangle\}\\
    &\mathcal{L}^1 v= w.
\end{aligned}
\end{equation*}
Note that $\mathcal{L}^1$ is the self-adjoint extension of the following multiplication operator for $u=(f,g)$
\begin{equation*}
    \mathcal{M} u= V_1 (f,g), \ \text{dom}(\mathcal{M})= \{ u \in \mathcal{E}(\mathcal{G}): V_1 u \in L^2(\mathcal{G}) \}. 
\end{equation*}
Indeed, for $v \in$ dom$(\mathcal{M})$ we have $v \in \widetilde{W}$, and we define $w= V_1 v \in L^2(\mathcal{G})$. Then for every $z \in \widetilde{W}$ we get $B^1(v,z)=\langle w,z\rangle$. Thus, $v \in $ dom$(\mathcal{L}^1)$ and $\mathcal{L}^1 v=w = V_1 v$. Hence, $\mathcal{M}\subset \mathcal{L}^1 $. Since $\mathcal{M}$ is self-adjoint, $\mathcal{L}^1 = \mathcal{M}$. The proof of the Theorem is complete. 
\end{proof}


\section{Proof of theorems \ref{FrobeN} and \ref{L2}}

Let us consider one \textit{a priori} positive single-lobe state $(\phi_c, \psi_c)$ solution for (\ref{sistema}) with $c \in \mathbb R$. For convenience, we denote $\phi = \phi_c$ and $\psi= \psi_c$. Thus, the linearized operator $\mathcal{L}_1$ in Theorem \ref{Oper} becomes as 
\begin{equation}
    \mathcal{L}_1= \text{diag}(- \partial_x^2 + (c-2)- \text{Log}|\phi|^2, - \partial_x^2 + (c-2)- \text{Log}|\psi|^2)
\end{equation}
with domain $\mathcal D=\{ (f,g) \in D_0: x^2g \in L^2([L,+\infty)) \}$. 

Next, for $(f,g) \in \mathcal D$ define $h(x)= g(x+L)$ for $x>0$. Then $h(0)=g(L)$ and $h'(0)=g'(L)$. Therefore, the eigenvalue problem $\mathcal{L}_1(f,g)^t = \lambda (f,g)^t$ is equivalent to 

\begin{equation}\label{peigen}
    \left\{   \begin{array}{ll}
    \mathcal{L}_{0,1} f(x) = \lambda f(x), \ \ \ x \in (-L,L), \\
     \mathcal{L}_{1,1} h(x)= \lambda h(x), \ \ \ x \in (0,+ \infty), \\ 
     (f,h) \in D_+, 
  \end{array}  \right.
\end{equation}
where 

\begin{equation}\label{L0+}
    \mathcal{L}_{0,1} = - \partial_x^2 + (c-2)- \text{Log}|\phi|^2, \ \ \ \mathcal{L}_{1,1} \equiv - \partial_x^2 + (c-2)- \text{Log}|\psi_{0,a}|^2, 
\end{equation}
and $\psi_{0,a}(x)=  e^{\frac{c+1}{2}}e^{-\frac{(x+a)^2}{2}}$, with $x>0$, $a>0$. In this form, 
\begin{equation}
\mathcal{L}_{1,1}= - \partial_x^2 + (x+a)^2-3.
\end{equation}
The domain $D_+$ is given by 
\begin{equation}
    D_+=\{ (f,h) \in X^2(-L,L): f(L)=f(-L)=h(0), f'(L)-f'(-L)=h'(0), x^2h \in L^2(0,+\infty)\}, 
\end{equation}
with $X^n(-L,L)\equiv H^n(-L,L) \oplus H^n(0,+\infty)$, $n \in \mathbb N$. We note that $(\phi, \psi_{0,a}) \in D_+$. 

For  notational convenience, let $\psi_a = \psi_{0, a}$. Define $\mathcal{L_+}= \text{diag}(\mathcal{L}_{0,1}, \mathcal{L}_{1,1})$ with domain $D_+$.  The proof of the Theorem \ref{FrobeN} will follow from sections 4.1 and 4.2 below.

 \subsection{Perron-Frobenius property and Morse index for   $(\mathcal L_{+}, D_+)$}

Initially, we show  that $ n(\mathcal L_{+})\geqq 1$. Since $(\phi, \psi_a )\in  D_+$ and 

\begin{equation}\label{nega}
\begin{array}{ll}
\langle \mathcal L_{+} (\phi, \psi_a)^t, (\phi, \psi_a)^t\rangle &= -2 \left[ \int_{-L}^L \phi^{2}(x)dx + \int_{0}^{+\infty} \psi_a^{2}(x) dx \right] <0,
\end{array}
\end{equation}
the mini-max principle implies $n(\mathcal L_{+})\geqq 1$. Now, we note that is possible to show, via the extension theory for symmetric operators of Krein-von Neumann, that $n(\mathcal L_{+})\leqq 2$.

\begin{theorem}\label{index} The Morse index associated to $(\mathcal L_{+}, D_+)$ is one. Consequently, the Morse index for $(\mathcal  L_1, \mathcal D)$  is also one.
\end{theorem}

The proof of Theorem \ref{index} is based on the following Perron-Frobenius property (PF property)  for $(\mathcal L_{+}, D_+)$. Our approach is grounded in EDO techniques (oscillation theorems) and the  extension theory for symmetric operators proposed in \cite{AC1, AC2} for  the NLS model \eqref{nlsp}. We note that significant adjustments to this approach are necessary  for the NLS-log model. The proof of Theorem \ref{index} is given at the end of this section.

\subsubsection{Perron-Frobenius property for   $(\mathcal L_{+}, D_+)$} 

We start  our analysis by defining the quadratic form $\mathcal Q$ associated to operator $\mathcal L_{+}$ on  $D_+$, namely, $\mathcal Q: D(\mathcal Q)\to \mathbb R$, with

\begin{equation}\label{Q}
\mathcal Q(\eta, \zeta) =\int_{-L}^L  (\eta')^2 + V_\phi \eta^2 dx + \int_{0}^{+\infty}  (\zeta')^2 + W_{\psi} \zeta^2  dx,
\end{equation}
$V_\phi=(c-2)-\text{Log}|\phi|^2 $, $W_\psi=(c-2)-\text{Log}|\psi_a|^2=(x+a)^2-3 $, and   $D(\mathcal Q)$ defined by 
 \begin{equation}\label{Qa}
   D(\mathcal Q)=\{(\eta,\zeta)\in X^1(-L,L): \eta(L)=\eta(-L)=\zeta(0), x \zeta \in L^2(0,+\infty)\}.
   \end{equation}

\begin{theorem}\label{Frobe} Let $\lambda_0<0$ be the smallest eigenvalue for $\mathcal L_{+}$ on  $D_+$ with associated eigenfunction $(\eta_{\lambda_0}, \zeta_{\lambda_0})$. Then, $\eta_{\lambda_0}$ and $\zeta_{\lambda_0}$ are positive functions. Moreover, $\eta_{\lambda_0}$ is even on $[-L,L]$.
\end{theorem}

\begin{proof} The strategy of the proof is based n tools used in \cite{AC1, AC2} for  the case of the NLS model in \eqref{nlsp}.  For the NLS-log model, significant changes  are required, and for the reader's convenience, we   highlight in the differences in the approach. We split the proof into several steps.

\begin{enumerate}
\item[1)] The profile $\zeta_{\lambda_0}$ is not identically zero: Indeed, suppose $\zeta_{\lambda_0}\equiv 0$, then $\eta_{\lambda_0}$ satisfies

\begin{equation}\label{trivial}
 \left\{ \begin{array}{ll}
\mathcal L_{0,1}\eta_{\lambda_0}(x)=\lambda_0 \eta_{\lambda_0} (x),\quad\quad x\in (-L,L),\\
 \eta_{\lambda_0}(L)=\eta_{\lambda_0}(-L)=0\\
 \eta'_{\lambda_0}(L)=\eta'_{\lambda_0}(-L).
  \end{array} \right.
 \end{equation} 
From the Dirichlet condition and oscillation theorems of the Floquet theory, $\eta_{\lambda_0}$ must be odd. By Sturm-Liouville theory, there is an eigenvalue $\theta$ for $\mathcal L_{0,1}$ such that $\theta<\lambda_0$, with an associated eigenfunction $ \xi>0$ on $(-L, L)$, and  $\xi(-L)=\xi(L)=0$. 
 

Let $\mathcal Q_{Dir}$ be the quadratic form associated to $\mathcal L_{0,1}$ with Dirichlet domain, {\it i.e.}, $\mathcal Q_{Dir}: H^1_0 (-L, L)\to \mathbb R$ defined by

\begin{equation}\label{Q_d}
\mathcal Q_{Dir}(f)=\int_{-L}^L (f')^2 + V_\phi f^2 dx.
 \end{equation} 
 
Then, $\mathcal Q_{Dir}(\xi)=\mathcal Q(\xi, 0)\geqq \lambda_0\|\xi\|^2$ and so, $\theta\geqq \lambda_0$. This is a contradiction.

\item[2)]  $\zeta_{\lambda_0}(0)\neq 0$:  suppose  $\zeta_{\lambda_0}(0)=0$ and we consider the odd-extension $\zeta_{odd}$ for $\zeta_{\lambda_0}$, and the even-extension $\psi_{even}$ of the tail-profile $\psi_a$ on all the line. Then, $\zeta_{odd}\in H^2(\mathbb R)$ and $\psi_{even}\in H^2(\mathbb R-\{0\})\cap H^1(\mathbb R)$. Next, we consider the following unfold operator  $\widetilde{\mathcal L}$ associated to $\mathcal L_{1,1}$,

\begin{equation}\label{Leven}
\widetilde{\mathcal L}=- \partial_x^2 + (c-2)- \text{Log}|\psi_{even}|^2=- \partial_x^2 +(|x|+a)^2-3,
 \end{equation} 
on the $\delta$-type interaction domain 
\begin{equation}\label{Ddelta}
   D_{\delta, \gamma}=\{f\in H^2(\mathbb R-\{0\})\cap H^1(\mathbb R): x^2f\in L^2(\mathbb R),  f'(0+)-f'(0-)=\gamma f(0)\}
   \end{equation}
for any $\gamma \in \mathbb R$. By the extension theory for symmetric operators, the family $(\widetilde{\mathcal L}, D_{\delta, \gamma})_{\gamma \in \mathbb R}$ represents all the self-adjoint extensions of the following symmetric operator $(\mathcal M_0, D(\mathcal M_0))$ defined by
$$
\mathcal M_0= \widetilde{\mathcal L},\quad D(\mathcal M_0)=\{f\in H^2(\mathbb R): x^2f\in L^2(\mathbb R), f(0)=0 \},
$$
 with deficiency indices  given by $n_{\pm}(\mathcal M_0)=1$. For clarity, we summarize these results (see  \cite{Ang2}). Consider  the Hilbert spaces scale associated with the self-adjoint operator
 $$
 \mathcal L=-\partial _x^2 +(|x|+a)^2,\quad D(\mathcal L)=\{f\in H^2(\mathbb R): x^2f\in  L^2(\mathbb R)\},
 $$
 {\it i.e}, $\mathcal H_s( \mathcal L)=\{f\in L^2(\mathbb R): \|f\|_{s,2}=\|(\mathcal L +I)^{s/2}f\|<\infty\}$, $s\geqq 0$, with 
 $$
 ...\subset\mathcal H_2( \mathcal L)\subset \mathcal H_1( \mathcal L)\subset L^2(\mathbb R)\subset\mathcal H_1( \mathcal L)\subset\mathcal H_{-2}( \mathcal L)\subset... \;.
 $$ 
 The dual space of $\mathcal H_s( \mathcal L)$ will be denoted by $\mathcal H_{-s}( \mathcal L)= \mathcal H_s( \mathcal L)'$. In this form, we get that the $\delta$-functional $\delta: \mathcal H_1( \mathcal L)\to \mathbb C$ defined by $\delta(\psi)=\psi(0)$ belongs to $\mathcal H_1( \mathcal L)'=\mathcal H_{-1}( \mathcal L)$  (by Sobolev embedding) and so $\delta\in \mathcal H_{-2}( \mathcal L)$. By Lemma 1.2.3 in \cite{AlKu},  the restriction $\mathcal M_0=\mathcal L-3$  to $D(\mathcal M_0)$ is a densely defined symmetric operator with deficiency indices $n_{\pm}(\mathcal M_0)=1$. Then, by Proposition \ref{11}  (Appendix) and Remark 4.2- item (iii) in \cite{Ang2}, we can characterized  all the self-adjoint extensions of  $(\mathcal M_0, D(\mathcal M_0))$ as $(\widetilde{\mathcal L}, D_{\delta, \gamma})_{\gamma \in \mathbb R}$.
 
  Now, the even tail-profile $\psi_{even}$  satisfies $\psi_{even}'(x)\neq 0$ for all $x\neq 0$, and so from the well-defined relation
\begin{equation}\label{M0}
\mathcal M_0 f=-\frac{1}{\psi'_{even}}\frac{d}{dx}\Big[(\psi'_{even})^2 \frac{d}{dx}\Big ( \frac{f}{\psi'_{even}}\Big)\Big],\quad x\neq 0
\end{equation}
we can see easily that $\langle \mathcal M_0 f, f\rangle\geqq 0$ for all $f\in D(\mathcal M_0)$. By extension theory  (Proposition \ref{semibounded}  in Appendix) we obtain that the Morse index for the family $(\widetilde{\mathcal L}, D_{\delta, \gamma})$ satisfies $n(\widetilde{\mathcal L})\leqq 1$, for all $\gamma \in \mathbb R$. Since $\zeta_{odd}\in D_{\delta, \gamma}$ (for any $\gamma$) and $\widetilde{\mathcal L}\zeta_{odd}=\lambda_0 \zeta_{odd}$ on $\mathbb R$, we have $n(\widetilde{\mathcal L})= 1$. Then,  $\lambda_0$ is the smallest negative eigenvalue for $\widetilde{\mathcal L}$ on $\delta$-interactions domains in \eqref{Ddelta}, and by Theorem \ref{PFpro}  in  Appendix (the Perron Frobenius property for $\widetilde{\mathcal L}$  with $\delta$-interactions domains on the line), $\zeta_{odd}$ must be positive, which is a contradiction. Therefore, $\zeta_{\lambda_0}(0)\neq 0$.

\item[3)] $\zeta_{\lambda_0}:[0, +\infty)\to \mathbb R$ can be chosen strictly positive: Without loss of generality, assume $\zeta_{\lambda_0}(0)>0$. The condition $\eta'_{\lambda_0}(L)-\eta'_{\lambda_0}(-L)=\zeta'_{\lambda_0}(0)$ implies
$$
\zeta'_{\lambda_0}(0)=\Big[ \frac{\eta'_{\lambda_0}(L)-\eta'_{\lambda_0}(-L)}{\zeta_{\lambda_0}(0)}\Big ]\zeta_{\lambda_0}(0)\equiv \gamma_0 \zeta_{\lambda_0}(0).
$$
Let $\zeta_{even}$ denote  even extension of $\zeta_{\lambda_0}$ to the entire  line. Then, $\zeta_{even}\in D_{\delta, 2\gamma_0}$ and $\widetilde{\mathcal L} \zeta_{even}=\lambda_0 \zeta_{even}$. A similar analysis as in item 2) above suggests that $\lambda_0$ is the smallest eigenvalue for $(\widetilde{\mathcal L}, D_{\delta, 2\gamma_0})$. Hence $\zeta_{even}$ is strictly positive on $\mathbb R$. Therefore, $\zeta_{\lambda_0}(x)>0$ for all $x\geqq 0$.

\item[4)] $\eta_{\lambda_0}:[-L,L]\to \mathbb R$ can be chosen strictly positive:  initially, we have that $\eta_{\lambda_0}$ satisfies the following boundary condition,

$$
\eta'_{\lambda_0}(L)-\eta'_{\lambda_0}(-L)=\Big[\frac{\zeta'_{\lambda_0}(0)}{\zeta_{\lambda_0}(0)}\Big] \zeta_{\lambda_0}(0)\equiv \alpha_{0} \zeta_{\lambda_0}(0)=\alpha_{0} \eta_{\lambda_0}(L).
$$

Consider the eigenvalue problem for $\mathcal L_{0,1}$ in \eqref{L0+} with  real coupled self-adjoint boundary condition determined by  $\alpha\in \mathbb R$:

\begin{equation}\label{RC}
(RC_{\alpha}):\;\; \left\{ \begin{array}{ll}
\mathcal L_{0,1}y(x)=\beta y(x),\quad\quad x\in (-L,L),\\
 y(L)=y(-L),\\
y'(L)-y'(-L)=\alpha y(L).
  \end{array} \right.
 \end{equation} 
By Theorem 1.35 in Kong\&Wu\&Zettl \cite{KWZ} or Theorem 4.8.1 in Zettl \cite{Z}, the first eigenvalue $\beta_0$ for \eqref{RC} with $\alpha=\alpha_0$ is simple. Since  $( \eta_{\lambda_0}, \lambda_0)$ solves  \eqref{RC},  $\lambda_0\geqq \beta_0$. {\it In the following we show} $\lambda_0= \beta_0$. Indeed,
we consider the quadratic form associated to the $(RC_{\alpha_0})$-problem  in \eqref{RC}, $\mathcal Q_{RC}$, where for $h\in H^1(-L,L)$ with $h(L)=h(-L)$,

\begin{equation}\label{QRC}
\mathcal Q_{RC}(h)=\int_{-L}^L (h')^2 + V_\phi h^2 dx-\alpha_0 |h(L)|^2.
 \end{equation} 
Let $\xi=\nu \zeta_{\lambda_0}$ with $\nu\in \mathbb R $  chosen such that $\xi(0)=\nu \zeta_{\lambda_0}(0)=h(L)$. Then, $(h,\xi)\in D(\mathcal Q)$ in \eqref{Qa}. Using $\mathcal L_{1,1} \zeta_{\lambda_0}=\lambda_0 \zeta_{\lambda_0}$, we obtain,

\begin{equation}\label{QQ}
\begin{aligned}
\mathcal Q_{RC}(h)&=\mathcal Q (h,\xi)- \alpha_0 h^2(L)- \int_{0}^{+\infty}(\xi')^2+ W_{\psi} \xi^2dx \\
&= \mathcal Q (h,\xi)- \alpha_0 h^2(L) + \nu^2 \zeta'_{\lambda_0}(0)\zeta_{\lambda_0}(0)- \lambda_0 \nu^2 \|\zeta_{\lambda_0}\|^2 \\
&=\mathcal Q (h,\xi)-\xi'_{\lambda_0}(0)h(L)+h(L)\xi'_{\lambda_0}(0)-\lambda_0  \|\xi\|^2 \\
&=\mathcal Q (h,\xi)-\lambda_0  \|\xi\|^2\geqq \lambda_0 [\|h\|^2+ \|\xi\|^2]-\lambda_0  \|\xi\|^2=\lambda_0  \|h\|^2.
\end{aligned} 
 \end{equation} 
Thus, $\beta_0\geqq \lambda_0 $ and so $\beta_0= \lambda_0 $. 

By the analysis above, we get that $ \lambda_0 $ {\it  is the first eigenvalue for the problem $(RC_{\alpha_0})$}  in \eqref{RC} and thus it is simple. Then, $\eta_{\lambda_0}$ is either odd or even. If $\eta_{\lambda_0}$ is odd,   the condition $\eta_{\lambda_0}(L)=\eta_{\lambda_0}(-L)$ implies $\eta_{\lambda_0}(L)=0$. However, $\eta_{\lambda_0}(L)= \zeta_{\lambda_0}(0)>0$. So, we must have that $\eta_{\lambda_0}$ is  even. Now, from  Oscillation Theorem for the  $(RC_{\alpha_0})$-problem, the number of zeros of  $\eta_{\lambda_0}$ on $[-L, L)$ is 0 or 1 (see Theorem 4.8.5 in \cite{Z}). Since $\eta_{\lambda_0}(-L)>0$ and $\eta_{\lambda_0}$ is even, we  necessarily conclude that $\eta_{\lambda_0}>0$ on $[-L,L]$. This completes the proof.
\end{enumerate}
\end{proof} 

\begin{corollary}\label{simples} Let $\lambda_0<0$ be the smallest  eigenvalue for $(\mathcal L_{+}, D_+)$. Then, $\lambda_0$ is simple.
\end{corollary}

\begin{proof}
The proof is immediate. Suppose $\lambda_0$ is a double eigenvalue. Then, there exists an eigenfunction   $(f_0, g_0)$ associated with $\lambda_0$ orthogonal to $(\eta_{\lambda_0},\zeta_{\lambda_0})$. By Theorem \ref{Frobe}, $f_0, g_0>0$. This contradicts the orthogonality of  eigenfunctions.
\end{proof}


 \subsubsection{Splitting eigenvalue method on tadpole graphs}

In the following, we establish our main strategy for studying eigenvalue problems on a tadpole graph $\mathcal G$ as deduced in Angulo \cite{AC1}. More exactly, we reduce the eigenvalue problem for $\mathcal L_{1}\equiv \text{diag}(\mathcal L_1^0, \mathcal L_1^1)$ in \eqref{L+}  to two classes of eigenvalue problems: one for $\mathcal L_1^0= - \partial_x^2 + (c-2)- \text{Log}|\phi|^2$ with periodic boundary conditions on $[-L,L]$ and the other for the operator $\mathcal L_1^1=- \partial_x^2 + (c-2)- \text{Log}|\psi|^2$ with  Neumann-type boundary conditions on $[L,+\infty)$. 

\begin{lemma}\label{split} Let us consider $(\mathcal L_{1}, \mathcal D)$  in Theorem \ref{FrobeN}. Suppose $(f,g)\in \mathcal D$ with  $g(L)\neq 0$ and  $\mathcal L_{1} (f,g)^t=\gamma (f,g)^t$, for $\gamma \leq 0$. Then, we obtain the following two eigenvalue problems: 

  \begin{equation*}\label{La}
 \left\{ \begin{array}{ll}
\mathcal L_1^0f(x)=\gamma f(x),\; x\in (-L,L),\\
 f(L)=(-L),\;\; f'(L)=f'(-L),
  \end{array} \right. \; \qquad
 \left\{ \begin{array}{ll}
\mathcal L_1^1g(x)=\gamma g(x),\; x>L,\\
g'(L+)=0.
  \end{array} \right.
 \end{equation*} 
\end{lemma} 

\begin{proof} For $(f, g)\in \mathcal D$ and $g(L)\neq 0$, we have $$f(-L)=f(L),\;\;f'(L)-f'(-L)=\left[\frac{g'(L+)}{g(L)} \right] f(L)\equiv \theta f(L),$$
and so $f$ satisfies the real coupled problem $(RC_\alpha)$ in \eqref{RC} with $\alpha=\theta$ and $\beta=\gamma$. Then, from the proof of Lemma 3.4 in \cite{AC1} we obtain that $\theta=0$, and thus we prove the Lemma. 
\end{proof}

\begin{remark} From Lemma \ref{split}, it follows that $\eta_{\lambda_0}$ in Theorem \ref{Frobe} remains an even-periodic function on $[-L, L]$ and $\zeta_{\lambda_0}$ satisfies a Neumann boundary condition on $[L, +\infty)$.
\end{remark}

 \subsubsection{Morse index for $(\mathcal L_{+}, D_+)$}

In the following, we provide the proof of Theorem \ref{index}.

\begin{proof} We consider $\mathcal L_{+}$ on $D_+$ and suppose $n(\mathcal L_{+})=2$ without loss of generality  (note that via extension theory we can show $n(\mathcal L_{+})\leqq 2$). From  Theorem \ref{Frobe} and Corollary \ref{simples}, the first negative eigenvalue $\lambda_{{0}}$ for $\mathcal L_{+}$ is simple with an associated eigenfunction $(\eta_{{\lambda_0}}, \zeta_{{\lambda_0}})$ having positive components and $\eta_{\lambda_0}$ being even on $[-L,L]$. Therefore, for $\lambda_{{1}}$ being the second negative  eigenvalue for $\mathcal L_{+}$, we need to have $\lambda_{1}>\lambda_{0}$. 

 Let $(f_{{1}}, g_{_{1}}) \in D_+$ be an associated eigenfunction to $\lambda_{{1}}$.  In the following, we divide our analysis into several steps.
 \begin{enumerate}
 \item[1)] Suppose $g_{{1}}\equiv 0$:  then $f_{{1}}(-L)=f_{{1}}(L)=0$ and  $f_{{1}}$ is odd (see step  1) in the proof of Theorem \ref{Frobe}). Now, our profile-solution $\phi$ satisfies   
 $$
 \mathcal L_{0,1} \phi'=0, \; \phi' \; \text{is odd}, \, \phi'(x)>0,\;\text{for}\; x\in [-L, 0),
 $$
   thus, since $\lambda_{1}<0 $ we obtain from the Sturm Comparison Theorem that there is $r\in (-L, 0)$ such that $\phi'(r)=0$, which is a contradiction. Then, $g_{1}$ is non-trivial. 
   
 \item[2)]   Suppose $g_{1}(0+)=0$: we consider the  odd-extension $g_{1, odd}\in H^2(\mathbb R)$ of $g_{1}$ and the unfold operator $\widetilde{\mathcal L}$ in \eqref{Leven} on the $\delta$-interaction domains $D_{\delta, \gamma}$ in \eqref{Ddelta}. Then, $g_{1, odd}\in D_{\delta, \gamma}$ for any $\gamma$ and so by Perron-Frobenius property for $(\widetilde{\mathcal L}, D_{\delta, \gamma})$ (see Appendix) we must have that $n( \widetilde{\mathcal L})\geqq 2$. But, by step $2)$ in the proof of Theorem \ref{Frobe} we obtain $n( \widetilde{\mathcal L})\leqq 1$ for all $\gamma$, and so we get a contradiction.
 
 \item[3)] Suppose $g_{1}(0+)>0$ (without loss of generality): we will see that $g_{{1}}(x)>0$ for all $x>0$. Indeed,  by Lemma \ref{split}
 we get  $g'_{{1}}(0+)=0$.  Thus, by considering  the even-extension $g_{1, even}$ of $ g_{{1}}$ on all the line and the unfold operator $\widetilde{\mathcal L}$  in \eqref{Leven} on  $D_{\delta, 0}$, we have that $g_{1, even}\in D_{\delta, 0}$ and so $n(\widetilde{\mathcal L})\geqq 1$. But, we know that  $n(\widetilde{\mathcal L})\leqq 1$ and so
  $\lambda_{{1}}$ is the smallest eigenvalue for  $(\widetilde{\mathcal L}, D_{\delta,0})$. Therefore, by the Perron-Frobenius property for $(\widetilde{\mathcal L}, D_{\delta, 0})$ (see Appendix) $g_{{1}}$ is strictly positive. We note that as $D_{\delta, 0}=\{f\in H^2(\mathbb R): x^2f\in L^2(\mathbb R)\}$, it follows from classical oscillation theory (see Theorem 3.5  in \cite{Berezin}) that  $g_{1, even}>0$.
 
 \item[4 )] Lastly,  since the pairs $(\zeta_{{\lambda_0}}, \lambda_{{0}})$ and $(g_1, \lambda_1)$ satisfy the eigenvalue problem
 \begin{equation}
  \left\{ \begin{array}{ll}
\mathcal L_{1,1}g(x)=\gamma g(x),\; x>0,\\
g'(0+)=0,
  \end{array} \right.
 \end{equation}
it follows the property that $\zeta_{{\lambda_0}}$ and $g_1$ need to be orthogonal (which is a contradiction). This finishes the proof. 
 \end{enumerate}
 \end{proof}

 \subsection{ Kernel for  $(\mathcal L_{1}, \mathcal D)$}

In the following, we study the nullity index for $\mathcal L_{1}$ on $\mathcal D$ (see item $4)$ in Theorem \ref{FrobeN}. Using the notation at the beginning of this section, it  is sufficient to show the following.

\begin{theorem}\label{ker}  Let us consider $\mathcal L_{+}=\text{diag}(\mathcal L_{0,1}, \mathcal L_{1,1})$ on  $D_+$. Then, the kernel associated with $\mathcal L_{+}$ on $D_+$ is trivial. 

\end{theorem}

\begin{proof} Let $(f, h)\in D_+$ such that $\mathcal L_{+}(f, h)^t= {\bf{0}}$. Thus, since $\mathcal L_{1,1}h=0$ and $\mathcal L_{1,1}\psi'_{a}=0$, we obtain from classical Sturm-Liouville theory on half-lines (\cite{Berezin}) that there is   $b\in \mathbb R$ with $h=b\psi'_{a}$ on $(0,+\infty)$. Next, we have the following cases:

\begin{enumerate}
\item[1)]  Suppose $b=0$:  then $h\equiv 0$ and $f$ satisfies $\mathcal L_{0,1} f=0$ with  Dirichlet-periodic conditions 
$$
f(L)=f(-L)=0,\;\;\text{and}\;\; f'(L)=f'(-L).
$$ 
Suppose $f\neq 0$. From the oscillation theory of Floquet theory and  Sturm-Liouville theory for Dirichlet conditions, $\mathcal L_{0,1} \phi'=0$ on $[-L, L]$, $ \phi'$  is odd and $\phi'(L)\neq 0$, we get  $f\equiv 0$ and so $ker(\mathcal L_{+})$ is trivial (see proof of Theorem 1.4 in \cite{AC1}).


\item[2)]  Suppose $b\neq 0$: then $h(0)\neq 0$ ($h(x)>0$ without loss of generality with $b<0$). Hence, from the splitting eigenvalue result in  Lemma \ref{split}, we obtain that $f$ satisfies
\begin{equation}\label{2ker}
  \left\{ \begin{array}{ll}
\mathcal L_{0,1}f(x)=0,\; x\in [-L,L],\\
 f(L)=f(-L)=h(0)>0,\;\;\text{and}\;\;  f'(L)=f'(-L),
  \end{array} \right.
 \end{equation}
and $h'(0)=0$. The last equality implies immediately $
\psi''_{a}(0)=0$,
therefore if we have for $\alpha= \psi''_a(0)$ that $\alpha \neq 0$, we obtain a contradiction. Then, $b=0$ and by item $1)$ above $ker(\mathcal L_{+})=\{{\bf {0}}\}$.

 Next, we consider the case $\alpha=0$. Then, initially, by Floquet theory and oscillation theory,  we have the following partial distribution of eigenvalues, $\beta_n$ and $\mu_n$, associated to $\mathcal L_{0,1}$ with periodic and Dirichlet conditions, respectively, 
\begin{equation}\label{inequal2}
\beta_0<\mu_0<\beta_1\leqq \mu_1\leqq \beta_2<\mu_2<\beta_3.
\end{equation}

 In the following, we will see  $\beta_1=0$ in \eqref{inequal2} and it is simple. Indeed, by \eqref{inequal2} suppose that $0> \mu_1$. Now, we know that  $\mathcal L_{0,1} \phi'=0$ on $[-L, L]$, $\phi'$ is odd and  $\phi'(x)>0$ for $[-L, 0)$, and the eigenfunction associated with $\mu_1$ is odd, therefore from the Sturm Comparison Theorem we get that $\phi'$ needs to have one zero on $(-L, 0)$, which is impossible. Hence, $ 0\leqq \mu_1$.  
 
 Next, suppose that $\mu_1=0$ and let $\chi_1$ be an odd eigenfunction for $\mu_1$. Let $\{\phi',  P\}$ be  a basis of solutions for the problem $\mathcal L_{0,1} g=0$ (we recall that $P$ can be chosen to be even and satisfying  $P(0)=1$ and  $P'(0)=0$). Then, 
  $\chi_1= a \phi'$ with $0=\chi_1(L)=a\phi'(L)$. Hence, $\chi_1\equiv 0$ which is not possible. Therefore,  $0<\mu_1$ and so $\beta_1=0$  is simple with eigenfunction $f$ (being even or odd). We note that  $\beta_2$ is also simple. 

 Lastly, since $f(-L)=f(L)>0$, it follows that $f$ is even and  by Floquet theory $f$ has exactly two different zeros $-a, a$ ($a>0$) on $(-L,L)$. Hence,  $f(0)<0$. Next, we consider the Wronskian function (constant) of $f$ and $\phi'$, 
$$
W(x)=f(x)\phi''(x)-f'(x)\phi'(x)\equiv C,\quad \text{for all}\;\; x\in [-L,L].
$$
 Then, $C=f(0)\phi''(0)>0$. Therefore, by the hypotheses ($\alpha=0$ ) we obtain
\begin{equation}\label{final}
C=f(L)\phi''(L)=h(0)\psi''_{a}(0)= 0,
\end{equation}
which is a contradiction. Then, $b=0$ and by item $1)$ above we get again $ker(\mathcal L_{+})=\{{\bf {0}}\}$. 
 This finishes the proof.
\end{enumerate}
\end{proof}

\subsection{Morse and nullity indices for operator $\mathcal{L}_2$}

\begin{proof}  $[${\bf{Theorem \ref{L2}}}$]$ We consider a positive single-lobe state $(\phi_c, \psi_c)$. Then, from  \eqref{sistema} we get $ (\phi_c, \psi_c)\in ker(\mathcal L_2)$. Next, we consider
$$
\mathcal M=-\partial_x^2+c-\text{Log}(\phi_c^{2}),\;\; \mathcal N=-\partial_x^2+c-\text{Log}(\psi_c^{2})
$$
then for any ${V}=(f,g)\in D_0$, we obtain
\begin{equation}\label{QQ}
\begin{aligned}
 \mathcal M f&=-\frac{1}{\phi_c}\frac{d}{dx}\Big[\phi_c^2 \frac{d}{dx}\Big ( \frac{f}{\phi_c}\Big)\Big],\quad\;\; x\in (-L, L)\\
 \\
  \mathcal N g&=-\frac{1}{\psi_c}\frac{d}{dx}\Big[\psi_c^2 \frac{d}{dx}\Big ( \frac{g}{\psi_c}\Big)\Big],\quad\;\;  x>L.
\end{aligned}
\end{equation}
Thus, we obtain immediately 
$$
\langle \mathcal L_2 {V}, {V}\rangle=\int_{-L}^L \phi_c^2 \Big(\frac{d}{dx}\Big( \frac{f}{\phi_c}\Big)\Big)^2 dx+  \int_{L}^{+\infty} \psi_c^2 \Big(\frac{d}{dx}\Big( \frac{g}{\psi_c}\Big)\Big)^2 dx\geqq 0.
$$
Moreover, since $\langle \mathcal L_2 {V}, {V}\rangle=0$ if and only if $f=d_1\phi_c$ and $g=d_2\psi_c$, we obtain from the continuity property at $x=L$ that $d_1=d_2$. Then, 
 $ker(\mathcal L_2)=\text{span}\{(\phi_c, \psi_c)\}$. This finishes the proof.
\end{proof}


 \section{Existence of the positive single-lobe state}

In this section, our focus is on proving the following existence theorem:

\begin{theorem}\label{Existencia}
    For any $c \in \mathbb R$, there is only one single lobe positive state $\Theta_c=(\phi_c, \psi_c) \in D_0$ that satisfies the NLS-log equation (\ref{nlslogv}), is monotonically decreasing on $[0, L]$ and $[L,+\infty)$, and the map $c\in \mathbb R  \mapsto \Theta_c \in D_0$ is $C^1$.  Moreover, the mass $\mu(c)=Q(\Theta_c)$ satisfies $\frac{d}{dc} \mu(c)>0$.
    \end{theorem}

For the proof of Theorem \ref{Existencia} we need some tools from dynamical systems theory for orbits on the plane and so for the convenience of the reader, we will divide our analysis into several steps (sub-sections) in the following.

\subsection{The period function}

We consider $L=\pi$ (without loss of generality) and $\Psi(x)= \psi(x+\pi)$, $x>0$. Then, the transformation 
\begin{equation} \label{transf}
    \phi_0(x)= e^{\frac{1-c}{2}} \Psi(x), \ \ \phi_1(x)= e^{\frac{1-c}{2}} \phi(x) 
\end{equation}
implies that system (\ref{sistema})with $Z=0$ is transformed into the following system of differential equations independent of the velocity $c$:  
\begin{equation}\label{1}
 \left\{ \begin{array}{ll}
  -\phi_1''(x)+  \phi_1(x)-\text{Log}(\phi_1^2(x))\phi_1(x)=0, \;\;\;\;x\in (-\pi,\pi),\\
   -\phi_0''(x)+  \phi_0(x)-\text{Log}(\phi_0^2(x))\phi_0(x)=0, \;\;\;\;x>0,\\
\phi_1( \pi)=\phi_1(-\pi) =\phi_0(0),\\
\phi_1'( \pi)-\phi_1'(-\pi)=\phi_0'(0).
  \end{array} \right.
 \end{equation}
We know  that a positive decaying solution to the equation $-\phi_0''+  \phi_0-\text{Log}(\phi_0^2)\phi_0=0$ on the positive half-line is expressed by 
\begin{equation}\label{sml}
    \phi_0(x)= e e^{\frac{-(x+a)^2}{2}}, \;\;\;\;x>0, \;\;a\in \mathbb R
\end{equation}  
with $ \phi_0(0)= e e^{\frac{-a^2}{2}}$. If $a>0$, $\phi_0$ is monotonically decreasing on $[0,+\infty)$ (a Gausson tail-profile); and if $a<0$, $\phi_0$ is non-monotone on $[0,+\infty)$ (a Gausson bump-profile). To prove Theorem \ref{Existencia}, we will only consider Gausson-tail profiles for $\phi_0$, therefore, we need to choose $a>0$.

The second-order differential equations in system \eqref{1} are integrable with a first-order invariant given by
\begin{equation}\label{2}
    E(\phi,\xi)= \xi^2 - A(\phi),\ \ \  \xi:= \frac{d\phi}{dx}, \ \ \ A(\phi)= 2\phi^2-log(\phi^2)\phi^2.
\end{equation}
Note that the value of $E(\phi,\xi)=E$ is independent of $x$. Next, for $A(u)=2u^2-{Log}(u^2)u^2$ we have that there is only one positive root of $A'(u)=2u(1-\text{Log}(u^2))$ denoted by $r_*$ such that $A'(r_*)=0$. In fact, $r_*=e^{\frac{1}{2}}$. As shown in Figure 4, for $E=0$ there are two homoclinic orbits: one corresponds to positive $u=\phi$ and the other to negative $u=-\phi$. Periodic orbits exist inside each of the two homoclinic loops and correspond to $E \in (E_*,0) $, where $E_*=-A(r_*)=-e$, and they correspond to either strictly positive $u$ or strictly negative $u$. Periodic orbits outside the two homoclinic loops exist for $E \in (0,+\infty)$ and they correspond to sign-indefinite $u$. Note that

\begin{equation}\label{energ}
    E+A(r_*)>0, \ \ \ E \in (E_*,+\infty).
\end{equation}
The homoclinic orbit with the profile solution (\ref{sml}) and $a>0$ corresponds to $\xi= -\sqrt{A(\phi)}$
for all $x>0$. Let us define $r_0:=  e e^{\frac{-a^2}{2}}$, that is, the value of $\phi_0(x)$ at $x=0$. Then, $-\sqrt{A(r_0)}$ is the value of $\phi'_0$ at $x=0$. Note that :


\begin{itemize}
    \item  $ r_0 \in   \left( 0,e \right)$ is a free parameter obtained from $a \in (0,+\infty)$.

    \item $r_0(a) \longrightarrow e  $ when $a \longrightarrow 0$.

    \item $r_0(a) \longrightarrow 0  $ when $a \longrightarrow +\infty$.
\end{itemize}

Hence, the profile $\phi_1$ will be found from the following boundary-value problem:
\begin{equation}\label{ep}
 \left\{ \begin{array}{ll}
  -\phi_1''(x)+  \phi_1(x)-\text{Log}(\phi_1^2(x))\phi_1(x)=0, \;\;\;\;x\in (-\pi,\pi),\\
\phi_1( \pi)=\phi_1(-\pi) =r_0,\\
\phi_1'( -\pi)=-\phi_1'(\pi)=\frac{\sqrt{A(r_0)}}{2}
  \end{array} \right.
 \end{equation}
where $ r_0 \in   \left( 0,e \right)$ will be considered as a free parameter of the problem. The positive single-lobe state $\phi_1$ will correspond to a part of the level curve $E(\phi, \xi) = E$ which intersects $r_0$ only twice at the ends of the interval $[-\pi, \pi]$ (see Figure 4 for a geometric construction of a positive single-lobe state for \eqref{ep}). 
\begin{figure}[h]
 	\centering
\includegraphics[angle=0,scale=0.12]{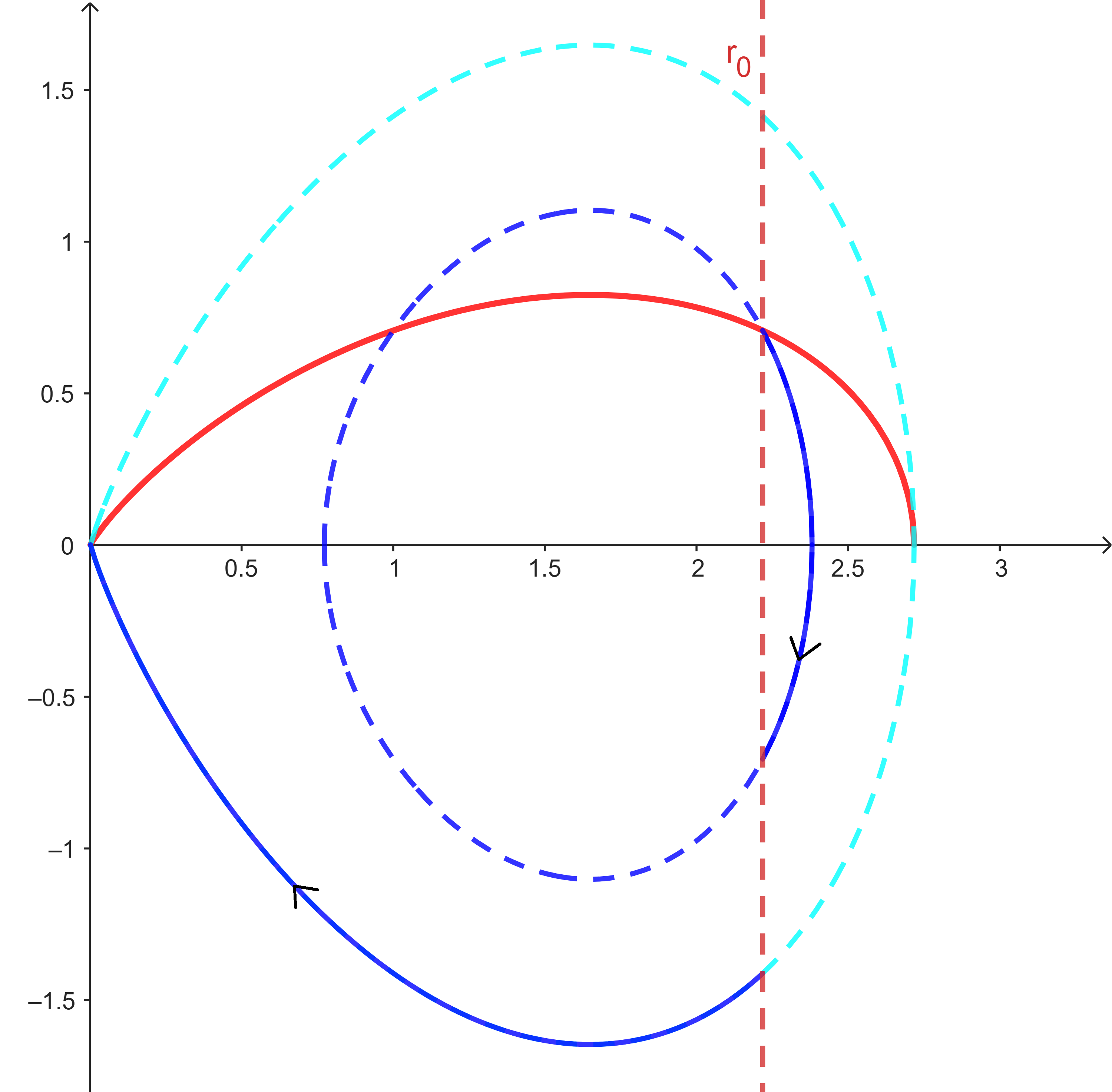} 
	\caption{$\cdot$  \tiny{Red line:  $s_0(r_0)=\frac12 \sqrt{A(r_0)}=\frac12 \phi_0'(0)$.\\
$\cdot$  Dashed-dotted vertical line depicts the value of $r_0 = \phi_0(0) = ee^{-a^2/2}$.\\ 
$\cdot$  The blue dashed curve represents the homoclinic orbit at $E = 0$, with the solid part depicting the shifted-tail  NLS-log soliton.\\
$\cdot$  The level curve $E(\phi,\zeta)=E(r_0,s_0)$, at $r_0= ee^{-a^2/2}$ and $s_0=\frac12 \sqrt{A(p_0)}$, is shown by the blue dashed line. \\
$\cdot$  The blue solid parts depict a suitable positive single-lobe state profile solution for the NLS-log.}}
\end{figure}

Let us denote by $s_0=\frac{\sqrt{A(r_0)}}{2}$ and we define the period function for a given $(r_0, s_0)$: 

\begin{equation}\label{fp}
    T_+(r_0,s_0):= \int_{r_0}^{r_+}\frac{d\phi}{\sqrt{E+A(\phi)}}, 
\end{equation}
where the value $E$ and the turning point $r_+$ are defined from $(r_0, s_0)$ by

\begin{equation} \label{definterv}
    E=s_0^2 - A(r_0)=-A(r_+).
\end{equation}
For each level curve of $E(\phi, \xi) = E$ inside the homoclinic loop we have $E \in (E_*,0)$, and  the turning point satisfies: 

\begin{equation}\label{pontos}
    0< r_* < r_+ < e. 
\end{equation}
We recall that $\phi_1$ is a positive single-lobe solution of the boundary-value problem (\ref{ep}) if and only if $r_0 \in (0,e)$ is a root of the nonlinear equation
\begin{equation} \label{eqphi1}
    T(r_0)= \pi, \ \ \ \text{where} \ \  T(r_0)= T_+\left( r_0,\frac{\sqrt{A(r_0)}}{2}\right).
\end{equation}
We note that, as $T (r_0)$ is uniquely defined by $r_0 \in \left(0, e\right)$, the nonlinear equation (\ref{eqphi1}) defines a unique mapping $\left(0, e\right) \ni r_0  \mapsto T(r_0) \in (0,+\infty)$. In the following sub-section, we show the monotonicity of this function. 

\subsection{Monotonicity of the period function}
 

In this sub-section, we will use the theory of dynamical systems  for orbits on the plane and the period function in \eqref{fp} to prove that the mapping $\left(0, e\right) \ni r_0  \mapsto T(r_0) \in (0,+\infty)$ is $C^1$ and monotonically decreasing.


Recall that if $W(\phi, \xi)$ is a $C^1$ function in an open region of $\mathbb{R}^2$, then the differential of $W$ is defined by
\begin{equation*}
    dW(\phi,\xi)=\frac{\partial W}{\partial \phi}d\phi + \frac{\partial W}{\partial \xi}d\xi
\end{equation*}
and the line integral of $dW(\phi,\xi)$ along any $C^1$ contour $\gamma$ connecting $(\phi_0,\xi_0)$ and $(\phi_1,\xi_1)$ does not depend on $\gamma$ and is evaluated as 

$$\int_\gamma dW(\phi,\xi)= W(\phi_1,\xi_1)- W(\phi_0,\xi_0).$$

At the level curve of $E(\phi, \xi) = \xi^2 - A(\phi) = E$, we can write

\begin{equation}\label{dif}
    d\left[ \frac{2(A(\phi)-A(r_*))\xi}{A'(\phi)}\right] = \left[2- \frac{2(A(\phi)-A(r_*))A''(\phi)}{\left[A'(\phi)\right]^2}\right]\xi d\phi +  \frac{2(A(\phi)-A(r_*))}{A'(\phi)} d\xi
\end{equation}
where the quotients are not singular for every $\phi > 0$. In view of the fact that $2\xi d\xi=A'(\phi)d\phi$ on the level curve $E(\phi, \xi) = E$, we can express (\ref{dif}) as: 

\begin{equation}\label{dif2}
    \frac{(A(\phi)-A(r_*))}{\xi} d\phi = -\left[2- \frac{2(A(\phi)-A(r_*))A''(\phi)}{\left[A'(\phi)\right]^2}\right]\xi d\phi +  d\left[ \frac{2(A(\phi)-A(r_*))\xi}{A'(\phi)}\right].
\end{equation}

The following lemma justifies the monotonicity of the mapping $\left(0, e\right) \ni r_0  \mapsto T(r_0) \in (0,+\infty)$.

\begin{lemma}\label{periodomono}
    The function $r_0  \mapsto T(r_0)$ is $C^1$ and monotonically decreasing for every $r_0 \in \left(0, e \right)$.  
\end{lemma}

\begin{proof}
    Since $s_0=\frac{\sqrt{A(r_0)}}{2}$ in (\ref{eqphi1}), for a given $r_0 \in \left(0, e\right)$, we have that the value of $T(r_0)$ is obtained from the level curve $E(\phi,\xi)=E_0(r_0)$, where 
\begin{equation}\label{energconst}
        E_0(r_0)\equiv  s_0^2 - A(r_0)= -\left( 1-\frac{1}{4} \right)A(r_0)=-\frac{3}{4}A(r_0).
    \end{equation}
For every $r_0 \in \left(0, e\right)$, we use the formula (\ref{dif2}) to get 

    \begin{equation}\label{funcionp}
    \begin{aligned}
        \left[E_0(r_0)+A(r_*)\right]T(r_0) &= \int_{r_0}^{r_+} \left[\xi - \frac{A(\phi)-A(r_*)}{\xi}  \right]d\phi \\
        &= \int_{r_0}^{r_+} \left[3 - \frac{2(A(\phi)-A(r_*))A''(\phi)}{\left[A'(\phi)\right]^2}  \right]\xi d\phi + \frac{2(A(r_0)-A(r_*))s_0}{A'(r_0)},
    \end{aligned}
    \end{equation}
    where we use (\ref{definterv}) to obtain that $\xi=0$ at $\phi= r_+$ and $\xi=s_0$ at $\phi=r_0$. Because the integrands are free of singularities and $E_0(r_0)+A(r_*) > 0$ due to (\ref{energ}), the mapping $\left(0, e\right) \ni r_0  \mapsto T(r_0) \in (0,+\infty)$ is $C^1$. We only need to prove that $T'(r_0)<0$ for every $r_0 \in \left(0, e\right)$. 
    
    Differentiating (\ref{funcionp}) with respect to $r_0$ yields

    \begin{equation} \label{difp}
    \begin{aligned}
    \left[E_0(r_0)+A(r_*)\right]T'(r_0) =& \left( [E_0(r_0)+A(r_*)]T(r_0) \right)' - E_0'(r_0)T(r_0) \\
    =& \int_{r_0}^{r_+} \left[3 - \frac{2(A(\phi)-A(r_*))A''(\phi)}{\left[A'(\phi)\right]^2}  \right]\frac{\partial \xi}{\partial r_0} d\phi - \left[3 - \frac{2(A(r_0)-A(r_*))A''(r_0)}{\left[A'(r_0)\right]^2}  \right]s_0 \\
    &+ \left[2- \frac{2(A(r_0)-A(r_*))A''(r_0)}{\left[A'(r_0)\right]^2}\right]s_0 + \frac{2(A(r_0)-A(r_*))}{A'(r_0)}s_0' - E_0'(r_0) \int_{r_0}^{r_+} \frac{d\phi}{\xi} \\
    =& -\frac{A(r_*)}{4s_0} - \frac{3A'(r_0)}{8}\int_{r_0}^{r_+}\left[1 - \frac{2(A(\phi)-A(r_*))A''(\phi)}{\left[A'(\phi)\right]^2}  \right] \frac{d\phi}{\xi},
    \end{aligned}
    \end{equation}
where we have used 

$$ s_0=\frac{2A(r_0)s_0'}{A'(r_0)}, \ \ \  E_0'(r_0)=-\frac{3A'(r_0)}{4},\ \ \ s_0'(r_0)= \frac{A'(r_0)}{8s_0},\ \ \ \frac{\partial \xi}{\partial r_0}= \frac{E_0'(r_0)}{2\xi}.$$

As $ A(\phi)=2\phi^2 - \text{Log}(\phi^2)\phi^2$, we 
obtain that (\ref{difp}) is as follows:

\begin{equation}\label{difpp}
[E_0(r_0)+A(r_*)]T'(r_0) = -\frac{A(r_*)}{4s_0} - \frac{3A'(r_0)}{8}\int_{r_0}^{r_+} \frac{\phi^2(3-\text{Log}(\phi^2)) - e(1+\text{Log}(\phi^2))}{\phi^2(1-\text{Log}(\phi^2))^2 \xi}d\phi.  
\end{equation}

Let us define the following function: 
\begin{equation}\label{numerador}
    f(\phi)=\phi^2(3-\text{Log}(\phi^2)) -e(1+\text{Log}(\phi^2)).
\end{equation}

\begin{figure}[h]
 	\centering
\includegraphics[angle=0,scale=0.3]{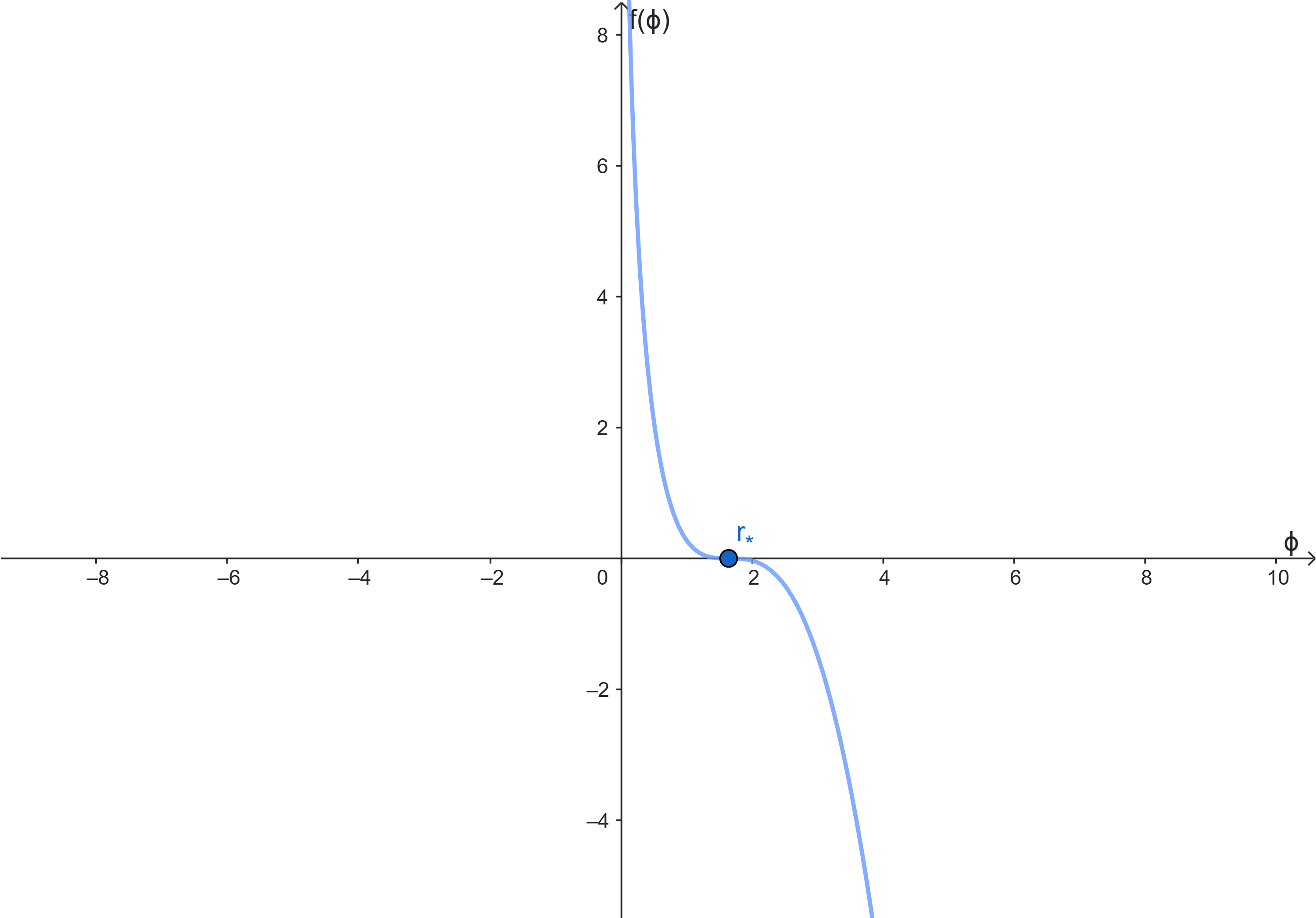}
\caption{Graph of $f$ in \eqref{numerador}}
\end{figure}

Differentiating (\ref{numerador}) with respect to $\phi$ yields: 
\begin{equation}
\begin{aligned}
    &f'(\phi)=2\phi(2-\text{Log}(\phi^2)) -\frac{2e}{\phi}\\
    \Leftrightarrow&\;\; \phi f'(\phi)=2\phi^2(2-\text{Log}(\phi^2)) -2e  \ \ \ \ \ \ \  (\phi>0).
\end{aligned}
\end{equation}

Note that $f'(\phi)=0 \Leftrightarrow \phi=r_*$ and $f'(\phi)<0$ for any $ \phi \in (0,r_*) \cup (r_*,+\infty)$. Since $f(\phi)=0 \Leftrightarrow \phi= r_*$ we have that $f(\phi)>0$ for any   $\phi \in (0,r_*)$ and $f(\phi)<0$ for any  $\phi \in (r_*,+\infty)$. 

Next, we know that $A'(r_0)<0$ for any $r_0 \in (r_*,e) $ and $f(\phi)<0$ for $\phi \in [r_0,r_+]\subset (r_*,e)$ then follows that  $T'(r_0)<0$. Similarly, since $A'(r_*)=0$, we have $T'(r_*)<0$.

Now, we consider the case $r_0 \in (0,r_*)$.  Then,
 \begin{equation}\label{suma}
 \begin{aligned}
     \int_{r_0}^{r_+} \frac{\phi^2(3-\text{Log}(\phi^2)) - e(1+\text{Log}(\phi^2))}{\phi^2(1-\text{Log}(\phi^2))^2 \xi}d\phi = I_1 + I_2
\end{aligned}
 \end{equation}
where 
\begin{equation}
    \begin{aligned}
        I_1 &= \int_{r_0}^{r_*} \frac{\phi^2(3-\text{Log}(\phi^2)) - e(1+\text{Log}(\phi^2))}{\phi^2(1-\text{Log}(\phi^2))^2 \xi}d\phi,\\
        I_2 &= \int_{r_*}^{r_+} \frac{\phi^2(3-\text{Log}(\phi^2)) - e(1+\text{Log}(\phi^2))}{\phi^2(1-\text{Log}(\phi^2))^2 \xi}d\phi
    \end{aligned}
\end{equation}
Because $f(\phi)>0$ for any   $\phi \in (0,r_*)$,  $I_1>0$.  By using $A'(\phi)=2\phi(1-\text{Log}(\phi^2)) $, we get that (see Figure 6)
\begin{equation}\label{dessup}
    \frac{\phi^2(3-\text{Log}(\phi^2)) - e(1+\text{Log}(\phi^2))}{\phi^2(1-\text{Log}(\phi^2))^2 } > \frac{A'(\phi)}{2\phi^3}
\ \  \text{for any   } \phi \in (r_*, e) \supset (r_*,r_+).
\end{equation}

\begin{figure}[h]
 	\centering
\includegraphics[angle=0,scale=0.2]{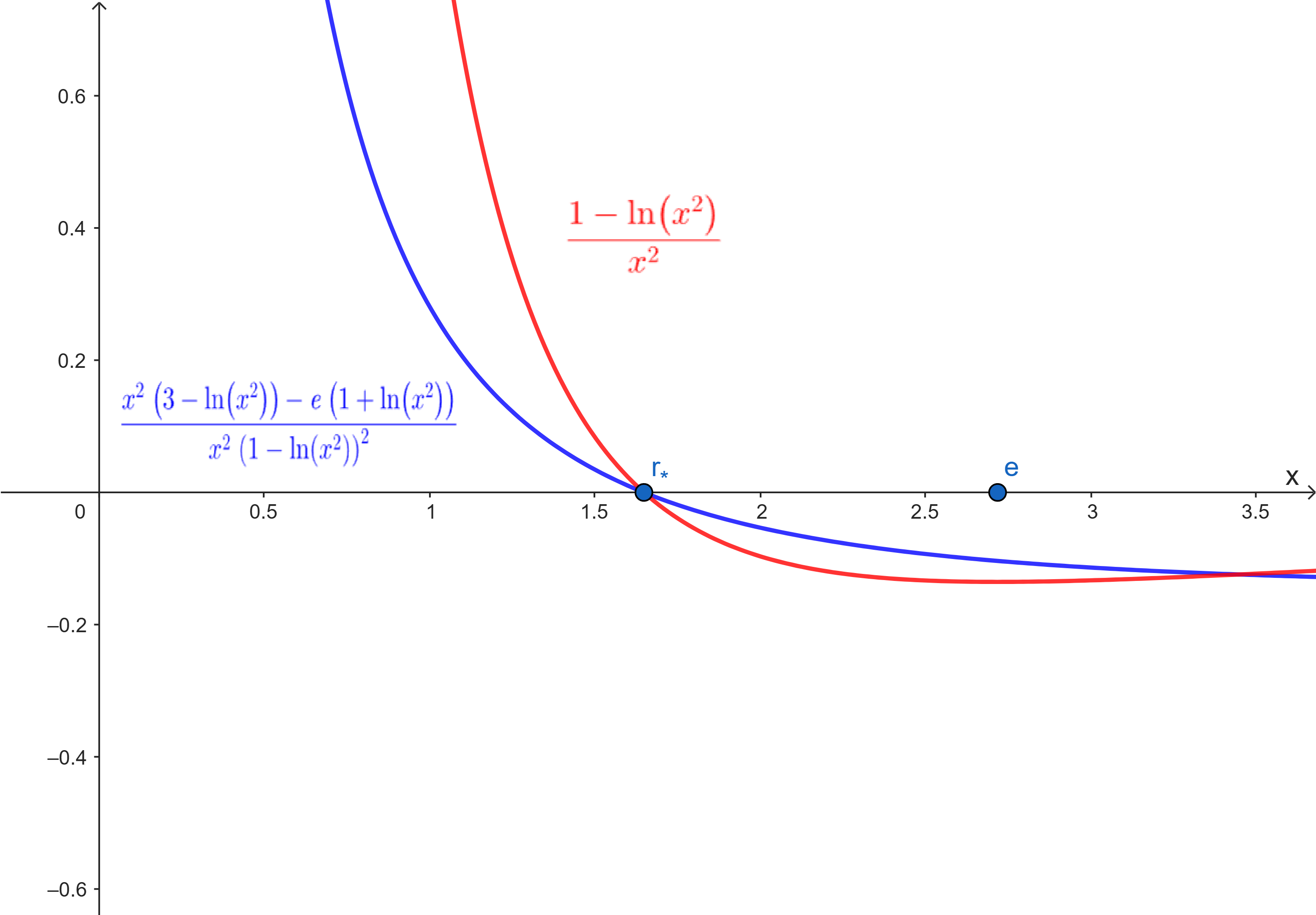}
\caption{Inequality  in \eqref{dessup}}
\end{figure}
By (\ref{dessup}) and proceeding by integration by parts, we have 

\begin{equation}\label{des}
    \begin{aligned}
        I_2 &> \int_{r_*}^{r_+} \frac{A'(\phi)}{2\phi^3\sqrt{E_0(r_0)+A(\phi)}}d\phi \\
        &= - \frac{ \sqrt{E_0(r_0)+A(r_*)} }{r_*^3} + 3\int_{r_*}^{r_+} \frac{\sqrt{E_0(r_0)+A(\phi)}}{\phi^4} d\phi\\
        &= - \frac{ \sqrt{E_0(r_0)+A(r_*)} }{r_*^3} + 3\int_{r_*}^{r_+} \frac{\xi }{\phi^4} d\phi
    \end{aligned}
\end{equation}

Substituting this into (\ref{difpp}) yields

\begin{equation}\label{difpdes}
\begin{aligned}
\left[E_0(r_0)+A(r_*)\right]T'(r_0) &= -\frac{A(r_*)}{4s_0} - \frac{3A'(r_0)}{8}I_1 - \frac{3A'(r_0)}{8}I_2  \\
&< -\frac{A(r_*)}{4s_0}+ \frac{3A'(r_0)}{8}\frac{ \sqrt{E_0(r_0)+A(r_*)} }{r_*^3} - \frac{3A'(r_0)}{8}I_1  \\ 
&\ \ \ -\frac{9A'(r_0)}{8} \int_{r_*}^{r_+} \frac{\xi }{\phi^4} d\phi.
\end{aligned}
\end{equation}

To evaluate the first two terms we will use that $A(r_*)=e$, $s_0= \frac{\sqrt{A(r_0)}}{2}$, and $\max_{r_0 \in (0,r_*)} A'(r_0)=4e^{-\frac{1}{2}}$, so that we get

\begin{equation}
    \begin{aligned}
        -\frac{A(r_*)}{4s_0}+ \frac{3A'(r_0)}{8}\frac{ \sqrt{E_0(r_0)+A(r_*)} }{r_*^3} 
        &= -\frac{e}{2\sqrt{A(r_0)}}+ \frac{3A'(r_0)}{8}\frac{ \sqrt{E_0(r_0)+e^c} }{e^{\frac{3}{2}}}\\
        &< -\frac{e}{2\sqrt{A(r_0)}}+ \frac{3\sqrt{E_0(r_0)+e} }{2e^{2}}.
    \end{aligned}
\end{equation}

For (\ref{energconst}) we get  

\begin{equation}\label{des1}
    \begin{aligned}
        -\frac{e}{2\sqrt{A(r_0)}}&+ \frac{3\sqrt{E_0(r_0)+e} }{2e^{2}}<0
        \Leftrightarrow \frac{3\sqrt{E_0(r_0)+e} }{2e^{2}} < \frac{e}{2\sqrt{A(r_0)}}\Leftrightarrow \frac{9(E_0(r_0)+e)}{4e^{4}} < \frac{e^{2}}{4A(r_0)} \\
        &\Leftrightarrow -\frac{27A(r_0)}{4} + 9e < \frac{e^{6}}{A(r_0)} \Leftrightarrow 9e < \frac{e^{6}}{A(r_0)} + \frac{27A(r_0)}{4}.
    \end{aligned}
\end{equation}

Since $A'(r_0)>0$ for any $r_0 \in (0,r_*)$, $A(r_0)$ increases monotonically for any $r_0 \in (0,r_*)$, and $\frac{1}{A(r_0)}$ decreases monotonically for any $r_0 \in (0,r_*)$. Differentiating $\frac{e^{6}}{A(r_0)} + \frac{27A(r_0)}{4}$ with respect to $r_0$ yields 

\begin{equation}
    \begin{aligned}
        \left( \frac{e^{6}}{A(r_0)} + \frac{27A(r_0)}{4}  \right)' &= -\frac{e^{6}A'(r_0)}{(A(r_0))^2}+ \frac{27A'(r_0)}{4} < -\frac{e^{6}A'(r_0)}{(A(r_*))^2}+ \frac{27A'(r_0)}{4} \\
        &= -e^{4}A'(r_0)+ \frac{27A'(r_0)}{4} = \left( \frac{27}{4} -e^4 \right) A'(r_0)<0.
    \end{aligned}
\end{equation}

Therefore, the function $\frac{e^{6}}{A(r_0)} + \frac{27A(r_0)}{4}$ is  monotonically decreasing for any $r_0 \in (0,r_*)$. Using this we get that

\begin{equation}
        \frac{e^{6}}{A(r_0)} + \frac{27A(r_0)}{4} > 
        \frac{e^{6}}{A(r_*)} + \frac{27A(r_*)}{4} 
        = \left( e^{4} + \frac{27}{4} \right) e 
        > 9e
\end{equation}

Thus, for (\ref{des1}) we get 

$$-\frac{A(r_*)}{4s_0}+ \frac{3A'(r_0)}{8}\frac{ \sqrt{E_0(r_0)+A(r_*)} }{r_*^3}< 0 \ \  \text{for any } r_0 \in (0,r_*).$$ 

As a result of the above calculations, for every $r_0 \in (0,r_*)$ we have $T'(r_0)<0$. This completes the proof. 
 
\end{proof}

\subsection{Proof of Theorem \ref{Existencia}}

\begin{proof}  Due to the monotonicity of the period function $T(r_0)$ in $r_0$ given by Lemma \ref{periodomono}  we have a diffeomorphism $\left(0, e \right) \ni r_0  \mapsto T(r_0) \in (0,+\infty)$.  In fact, we will show that $T(r_0)\longrightarrow 0$ when $r_0 \longrightarrow e$ and $T(r_0)\longrightarrow +\infty$, when $r_0 \longrightarrow 0$.  Indeed,  from (\ref{definterv}), (\ref{energconst}) and (\ref{pontos}), we have  that for  $r_0 \in (0, e)$, the equation $E_0(r_0) = -A(r_+)$ determines $r_+=r_+(r_0)$ from the nonlinear equation:
\begin{equation}\label{defppost}
        2r_+^2-\text{Log}(r_+^2)r_+^2 = \frac{3}{4}[2r_0^2-\text{Log}(r_0^2)r_0^2],
    \end{equation}
 and $\lim_{r_0 \rightarrow e} \frac{3}{4}[2r_0^2-\text{Log}(r_0^2)r_0^2]=0$. Now for (\ref{defppost}) we have 

    \begin{equation}
        \begin{aligned}
            2r_+^2-\text{Log}(r_+^2)r_+^2 \rightarrow 0, \text{ when} \ r_0 \rightarrow e &\Leftrightarrow r_+ \rightarrow e , \text{ when} \ r_0 \rightarrow e\\
            &\Leftrightarrow |r_+ - r_0|\rightarrow 0, \text{ when} \ r_0 \rightarrow e.
        \end{aligned}
    \end{equation}
Since the weakly singular integrand below is integrable, we have 
\begin{equation}
        T(r_0)= \int_{r_0}^{r_+} \frac{d\phi}{\sqrt{E+A(\phi)}} = \int_{r_0}^{r_+} \frac{d\phi}{\sqrt{A(\phi)-A(r_+)}} \rightarrow 0, \text{ when} \ r_0 \rightarrow e,
    \end{equation}
    Next,  for every $0 < r_0 < r_+ < e$ we obtain
\begin{equation}
    \begin{aligned}
        T(r_0)= \int_{r_0}^{r_+}& \frac{d\phi}{\sqrt{A(\phi)-A(r_+)}} \geq \int_{r_0}^{r_+} \frac{d\phi}{\sqrt{A(\phi)}}\\
        &= \int_{r_0}^{r_+} \frac{d\phi}{\sqrt{2\phi^2 -\text{Log}(\phi^2)\phi^2}} = \int_{r_0}^{r_+} \frac{d\phi}{\phi \sqrt{2-\text{Log}(\phi^2)}}.
    \end{aligned}
    \end{equation}
    
    Since $r_+ \in (r_*, e)$ and 
    $$
    \lim_{r_0 \rightarrow 0} \frac{3}{4}[2r_0^2-\text{Log}(r_0^2)r_0^2]=0,
    $$ 
    by (\ref{defppost}) we have $r_+ \rightarrow e , \text{ when} \ r_0 \rightarrow 0$. Therefore, as 
    $$
    \int_{0}^{e} \frac{d\phi}{\phi \sqrt{(2-\text{Log}(\phi^2)}} = \lim_{\phi \rightarrow 0}\sqrt{2-\text{Log}(\phi^2)}=+\infty
    $$ 
    we have $T(r_0) \rightarrow +\infty$ as $r_0 \rightarrow 0$. Then, as the function $T(r_0)$ is monotone decreasing, the codomain of the function $r_0  \mapsto T(r_0)$ is indeed $(0,+\infty)$.

Thus, there is a unique $r_0\in (0, e)$ such that $T(r_0)=\pi$. Therefore, the boundary-value problem \eqref{ep} has a solution. Next, define $a_0>0$ such that $ e e^{\frac{-a_0^2}{2}}=r_0$. Then,
$$
\phi_c(x)=e^{\frac{c-1}{2}}\phi_1(x)\;\;\text{and}\;\; \psi_c(x)=e^{\frac{c+1}{2}}e^{\frac{-(x-\pi+a_0)^2}{2}}
$$
satisfies (\ref{sistema}) with $Z=0$. Obviously, we have that  $ \mathbb R \ni c \mapsto \Theta(c)=(\phi_c,\psi_c) \in D_0$ is a $C^1$-mapping of positive single-lobe state for the NLS-log equation on the tadpole graph.

Next, the relation
$$
\mu(c)=\|\Theta(c)\|^2= e^{c-1}\Big[\int_{-\pi}^\pi \phi_1^2(x)dx + \int_{0}^{+\infty} \phi_0^2(x)dx\Big]
$$
  with $   \phi_0(x)=ee^{\frac{-(x+a_0)^2}{2}}$ and $\phi_1$ independent of $c$, implies $\mu'(c)>0$ for all $c\in \mathbb R$. The proof of the theorem is completed.

\end{proof}

\section{Proof of the stability theorem }

In this section, we show Theorem \ref{stability} based on the stability criterion in Theorem \ref{main} (Appendix).

\begin{proof} By Theorem \ref{Existencia} we obtain the existence of a $C^1$-mapping $c \in \mathbb R \rightarrow \Theta_c= (\phi_c, \psi_c)$ of positive single-lobe states for the NLS-log model on a tadpole graph. Moreover, the mapping $c\to \| \Theta_c\|^2$ is strictly increasing.  Next,  from Theorem \ref{FrobeN}  we get that the Morse index for $\mathcal L_1$ in \eqref{L+}, $n(\mathcal L_{1})$, satisfies $n(\mathcal L_{1})=1$ and  $ker(\mathcal L_{1})=\{\bf{0}\}$.  Moreover, Theorem \ref{L2} establishes that $Ker(\mathcal L_{2})=span\{(\phi_c, \psi_c)\}$ and $\mathcal L_{2}\geqq 0$. Thus, by Theorem \ref{global} and stability criterion in Theorem \ref{main}, we obtain that  $e^{ict}
(\phi_c, \psi_c)$ is orbitally stable in $\widetilde{W}$. This finishes the proof.
\end{proof}

\section{Discussion and open problems}
\label{secconclu}
In this paper, we have established the existence and orbital stability of standing wave solutions for the NLS-log model on a tadpole graph with a profile being a positive single-lobe state. To that end, we use  tools from dynamical systems theory for orbits on the plane and we use the period function for showing the existence of  such a state with Neumann-Kirchhoff condition at the vertex $\nu=L$ ($Z = 0$ in \eqref{sistema}).  We believe that the existence (and simultaneous stability properties) of positive single-lobe solutions on the tadpole graph can be obtained via variational analysis applied to the constrained problem  
\[
I_\lambda = \inf \left\{ E(U) : U \in W(\mathcal{G}), \; Q(U) = \lambda > 0 \right\},
\]  
 For this analysis, we would to use the approach in Cazenave \cite{Caz83} combined with symmetric rearrangement strategies on metric graphs (see Adami {\it et al.} \cite{ACFN} or  Ardila \cite{Ar0}).  
Furthermore, we note that via this approach for the existence, the phase velocity $c$ in the vectorial NLS-log equation \eqref{nlslogv} is determined by the Lagrange multiplier associated with the minimization problem $I_\lambda$. Additionally, stability information (after proving the globally well-posed theory in $W(\mathcal{G})$)  is only established for the following minimizing set  
\[
G_\lambda = \left\{ U \in W(\mathcal{G}) : E(U) = I_\lambda, \; Q(U) = \lambda > 0 \right\}.
\]  
Thus, the advantage of our approach using the period function is that we can demonstrate the existence of positive single-lobe states for any $c \in \mathbb{R}$. 

The orbital stability of the positive single-lobe states established in Theorem \ref{Existencia}  is based on the framework of Grillakis {\it et al.}  in  \cite{GrilSha87} adapted to the tadpole graph and so via a splitting eigenvalue method and tools of the extension theory of Krein-von Neumann for symmetric operators and  the Sturm Comparison Theorem we identify the Morse index and the nullity index of a specific linearized operator around a positive single-lobe state which is a fundamental ingredient in this endeavor. For the case $Z\neq 0$ in \eqref{sistema} and by supposing the existence of a positive single-lobe state, is possible to obtain similar results for $(\mathcal L_2, D_Z)$ as in Theorem \ref{L2}.  Statements $1)-2)-3)$ in Theorem \eqref{FrobeN} are also true for $(\mathcal L_1, D_Z)$ (see Section 3 in \cite{AC2}).  Moreover, if we define the quantity for the shift $a=a(Z)$
$$
\alpha(a)=\frac{\psi''_c(L)}{\psi'_c(L)}+Z= \frac{1-a^2}{a}+Z,
$$ 
then, the kernel associated with $(\mathcal L_{1}, D_Z)$ is trivial in the following cases: for $\alpha\neq 0$ or $\alpha= 0$ in the case of admissible parameters $Z$  satisfying  $Z\leqq 0$ (see Theorem 1.3 and Lemma 3.5 in \cite{AC2}). The existence of these positive single-lobe state profiles with $Z\neq 0$ is more challenging and will be addressed in future work.  Our approach has a prospect of being extended to study stability properties of other standing wave states for the NLS-log on a tadpole graph (by instance, to choose boundary conditions in the family of 6-parameters given in \eqref{6bc}) or on another non-compact metric graph such as a looping edge graph (see Figure 7 below), namely,  a graph consisting of a circle with several half-lines attached at a single vertex (see \cite{AC1, AC2, AM}).
\begin{figure}[htp]\label{figure2}
\centering
\begin{tikzpicture}[scale=1.8]

\draw [thin, dashed] (1.3,0)--(1,0);
\draw[stealth-](1,0)--(0,0);
\node at (1.4,0.1){};

\node at (-0.15,0.1){\tiny{$-L$}};

\node at (-0.15,-0.1){\tiny{$+L$}};

\draw[stealth-](0.7,0.7)--(0,0);
\draw [thin, dashed] (0.7,0.7)--(1,1);
\node at (0.9,1.1)[rotate=45]{};

\draw[stealth-](-0.45,-0.5)--(-0.5,-0.5);

\draw[-stealth](-0.45,0.5)--(-0.5,0.5);

\draw[stealth-](-1,0.0)--(-1.0,0.05);

\draw[stealth-](0.7,-0.7)--(0,0);
\draw [thin, dashed] (0.7,-0.7)--(1,-1);
\node at (1.1,-0.9)[rotate=-45]{};

\fill (0,0)  circle[radius=1pt];

\draw [thin] (-0.5,0) circle(0.5);

\draw[-stealth](0,0)--(0.89,0.45);
\node at (1.19,0.75)[rotate=3
0]{};
\draw [thin, dashed] (0.89,0.45)--(1.19,0.6);
\draw[-stealth](0,0)--(0.89,-0.45);
\node at (1.23,-0.49)[rotate=-30]{};
\draw [thin, dashed] (0.89,-0.45)--(1.19,-0.6);
\fill (0,0)  circle[radius=1pt];
\end{tikzpicture}
\centerline{Figure 7: A looping edge graph  with $N=5$ half-lines}
\end{figure}

 \section{Appendix}
 
\subsection{Classical extension theory results}

The following results are classical in the extension theory of symmetric operators and can be found in \cite{Nai67, RS}. Let $A$ be a closed densely defined symmetric operator in the Hilbert space $H$. The domain of $A$ is denoted by $D(A)$. The deficiency indices of $A$ are denoted by  $n_\pm(A):=\dim \text{ker}(A^*\mp iI)$, with $A^*$ denoting the adjoint operator of $A$.  The number of negative eigenvalues counting multiplicities (or Morse index) of $A$ is denoted by  $n(A)$. 
 
\begin{theorem}[von-Neumann decomposition]\label{d5} 
Let $A$ be a closed, symmetric operator, then
\begin{equation}\label{d6}
D(A^*)=D(A)\oplus\mathcal N_{-i} \oplus\mathcal N_{+i}.
\end{equation}
with $\mathcal N_{\pm i}= \ker (A^*\mp iI)$. Therefore, for $u\in D(A^*)$ and $u=x+y+z\in D(A)\oplus\mathcal N_{-i} \oplus\mathcal N_{+i}$,
\begin{equation}\label{d6a}
A^*u=Ax+(-i)y+iz.
\end{equation}
\end{theorem}

\begin{remark} The direct sum in (\ref{d6}) is not necessarily orthogonal.
\end{remark}

\begin{proposition}\label{11}
	Let $A$ be a densely defined, closed, symmetric operator in some Hilbert space $H$ with deficiency indices equal to  $n_{\pm}(A)=1$. All self-adjoint extensions $A_\theta$ of $A$ may be parametrized by a real parameter $\theta\in [0,2\pi)$ such that
	\begin{equation*}
	\begin{split}
	D(A_\theta)&=\{x+c\phi_+ + \zeta e^{i\theta}\phi_{-}: x\in D(A), \zeta \in \mathbb C\},\\
	A_\theta (x + \zeta \phi_+ + \zeta e^{i\theta}\phi_{-})&= Ax+i \zeta \phi_+ - i \zeta e^{i\theta}\phi_{-},
	\end{split}
	\end{equation*}
	with $A^*\phi_{\pm}=\pm i \phi_{\pm}$, and $\|\phi_+\|=\|\phi_-\|$.
\end{proposition}

The following proposition provides a strategy for estimating the Morse-index of the self-adjoint extensions (see \cite{Nai67}, \cite{ RS}-Chapter X).

\begin{proposition}\label{semibounded}
Let $A$  be a densely defined lower semi-bounded symmetric operator (that is, $A\geq mI$)  with finite deficiency indices, $n_{\pm}(A)=k<\infty$,  in the Hilbert space ${H}$, and let $\widehat{A}$ be a self-adjoint extension of $A$.  Then the spectrum of $\widehat{A}$  in $(-\infty, m)$ is discrete and consists of, at most, $k$  eigenvalues counting multiplicities.
\end{proposition}

The next Proposition can be found in Naimark \cite{Nai67} (see Theorem 9).

\begin{proposition}
\label{esse}
All self-adjoint extensions of a closed, symmetric operator
which has equal and finite deficiency indices have one and the
same continuous spectrum.
\end{proposition}

 \subsection{Perron-Frobenius property for $\delta$-interaction Schr\"odinger operators on the line}

In this section, we establish  the Perron-Frobenius property for the unfolded self-adjoint operator $ \widetilde{\mathcal L}$ in  \eqref{Leven},
\begin{equation}\label{2Leven}
\widetilde{\mathcal L}=-\partial_x^2+ (c-2)- \text{Log}(\psi_{even}^2)= -\partial_x^2+ (|x|+a)^2-3
 \end{equation} 
 on $\delta$-interaction domains, namely,
\begin{equation}\label{2Ddelta}
   D_{\delta, \gamma}=\{f\in H^2(\mathbb R-\{0\})\cap H^1(\mathbb R): x^2f \in L^2(\mathbb R), f'(0+)-f'(0-)=\gamma f(0)\}
   \end{equation}
for any $\gamma \in \mathbb R$.  Here, $\psi_{even}$ is the even extension to the whole line of the Gausson tail-soliton profile $\psi_{a}(x)=  e^{\frac{c+1}{2}}e^{-\frac{(x+a)^2}{2}}$, with $x>0$, $a>0$. Since  
$$
\lim\limits_{|x|\to +\infty} (|x|+a)^2=+\infty,
$$
operator $\widetilde{\mathcal L}$ has a discrete spectrum, $\sigma(\widetilde{\mathcal L})=\sigma_d(\widetilde{\mathcal L})=\{\lambda_{k}\}_{k\in\mathbb{N}}$ (this statement can be obtained similarly using the strategy in the proof of Theorem 3.1 in  \cite{Berezin}). In particular, from subsections 2-3 in Chapter 2 in \cite{Berezin} adapted to $(\widetilde{\mathcal L}, D_{\delta, \gamma})$ for $\gamma$ fixed (see also Lemma 4.8 in \cite{Ang2}) we have the following distribution of the eigenvalues $\lambda_0<\lambda_1<\cdot\cdot\cdot<\lambda_k<\cdot\cdot\cdot $,
with $\lambda_k\to +\infty$ as $k\to +\infty$ and from the semi-boundedness of $V_a=(|x|+a)^2-3$   we obtain that any solution of the equation $\widetilde{\mathcal L} v=\lambda_k v$,  $v\in D_{\delta, \gamma}$,
 is unique up to a constant factor. Therefore each eigenvalue $\lambda_{k}$ is simple.

\begin{theorem}\label{PFpro}$[$ Perron-Frobenius property $]$ Consider the family of self-adjoint operators $(\widetilde{\mathcal L}, D_{\delta, \gamma})_{\gamma \in \mathbb R}$. For $\gamma$ fixed, let  $\lambda_0=\inf \sigma_p(\widetilde{\mathcal L})$ be the smallest eigenvalue. Then, the corresponding eigenfunction $\zeta_0$ of $\lambda_0$ is positive (after replacing $\zeta_0$ by $-\zeta_0$ if necessary) and even.
\end{theorem}

\begin{proof} This result can be obtained by following the strategy in the proof of Theorem 3.5 in \cite{Berezin}. Here, we give another approach via a slight twist of standard abstract Perron-Frobenius arguments (see Proposition 2 in Albert\&Bona\&Henry \cite{ABH}). The basic point in the analysis is to show that the Laplacian operator $-\Delta_\gamma \equiv-\frac{d^2}{dx^2}$ on the domain $D_{\delta, \gamma}$ has its resolvent $R_\mu=(-\Delta_\gamma +\mu)^{-1}$ represented by a positive kernel for some $\mu>0$ sufficiently large. Namely, for $f\in L^2(\mathbb R)$
$$
R_\mu f(x)= \int_{-\infty}^{+\infty}K(x, y) f(y)dy
$$
with $K(x, y)>0$ for all $x,y\in \mathbb R$. By the convenience of the reader, we show this main point; the remainder of the proof follows the same strategy as in \cite{ABH}. Thus, for $\gamma$ fixed, let $\mu>0$ be sufficiently large (with $-2\sqrt{\mu}<\gamma$ in the case $\gamma<0$), then from the Krein formula (see Theorem 3.1.2 in \cite{Albe}) we obtain
$$
K(x,y)=\frac{1}{2\sqrt{\mu}}\Big[ e^{-\sqrt{\mu}|x-y|}-\frac{\gamma}{\gamma+2\sqrt{\mu}} e^{-\sqrt{\mu}(|x|+|y|)}\Big].
$$
Moreover, for every $x$ fixed, $K(x, \cdot) \in L^2(\mathbb R)$. Thus, the existence of the integral above is guaranteed by Holder's inequality. Moreover, $x^2 R_\mu f\in  L^2(\mathbb R)$.  Now, since $K(x,y)=K(y,x)$, it is sufficient to show that $K(x,y)>0$ in the following cases.
\begin{enumerate}
\item[(1)] Let $x>0$ and $y>0$ or $x<0$ and $y<0$: for $\gamma \geqq 0$, we obtain from $\frac{\gamma}{\gamma+2\sqrt{\mu}}<1$ and $|x-y|\leqq |x|+ |y|$, that $K(x,y)>0$. For $\gamma <0$ and $-2\sqrt{\mu}<\gamma$,  it follows immediately $K(x,y)>0$.

\item[(2)] Let $x>0$ and $y<0$: in this case,
$$
K(x,y)= \frac{1}{\gamma+2\sqrt{\mu}}e^{-\sqrt{\mu}(x-y)} >0
$$
for any value of $\gamma$ (where again  $-2\sqrt{\mu}<\gamma$ in the case $\gamma<0$) .
\end{enumerate}

This finishes the proof.
\end{proof}

\subsection{Orbital stability criterion}
 
For the convenience of the reader,  in this subsection we adapt the abstract stability results from Grillakis\&Shatah\&Strauss in  \cite{GrilSha87} for the case of the NLS-log on a tadpole graph.  This criterion was used in the proof of Theorem \ref{stability} for the case of standing waves that are positive single-lobe states.

\begin{theorem}\label{main} Suppose that there is  $C^1$-mapping $c\to (\phi_{c}, \psi_{c})$ of standing-wave solutions for the NLS-log model \eqref{nlslog} on a tadpole graph. We consider the operators $\mathcal L_{1}$ and $\mathcal L_{2}$ in \eqref{L+}.
For $\mathcal L_{1}$ suppose that the  Morse index is one and its kernel is trivial. For $\mathcal L_{2}$ suppose that it is a non-negative operator with kernel generated by the profile $(\phi_{c}, \psi_{c})$. Moreover, suppose that the Cauchy problem associated with the NLS-log model \eqref{nlslog} is globally well-posed in the space $\widetilde{W}$ in \eqref{WW}. Then, $e^{ic t}
(\phi_{c}, \psi_{c})$ is orbitally stable in $\widetilde{W}$  if $ \frac{d}{dc} ||(\phi_{c}, \psi_{c})||^2>0$. 
\end{theorem}

\vskip0.1in

 \noindent
{\bf Data availability} 
\vskip0.1in
No data was used for the research described in the article.
 
 \vskip0.1in

 \noindent
{\bf Acknowledgements.}  The authors thank the anonymous referees for their insightful comments and suggestions which improved the quality of the paper. J. Angulo was partially funded by CNPq/Brazil Grant and Universal/CNPq project, him would like to thank to Mathematics Department of the Federal University of Pernambuco (UFPE) by the support and the warm stay during  the Summer-program/2025, where part of the present project was developed. A. P\'erez was partially funded by CAPES/Brazil and UNIVESP/S\~ao Paulo/Brazil.

\end{document}